\documentclass[a4paper,10pt]{article}
\usepackage[T1,T2A]{fontenc}
\usepackage[cp1251]{inputenc}
\usepackage[english]{babel}
\usepackage{amsmath}
\usepackage{amsfonts}
\usepackage{amssymb}
\usepackage{mathrsfs}
\usepackage{amsthm}
\usepackage{titling}
\usepackage{graphicx}
\usepackage{wrapfig}

\newtheorem{thm}{Theorem}[section]
\newtheorem{cor}[thm]{Corollary}
\newtheorem{lem}[thm]{Lemma}
\newtheorem{prop}[thm]{Proposition}
\theoremstyle{definition}
\newtheorem{defn}[thm]{Definition}

\theoremstyle{remark}
\newtheorem{rem}[thm]{Remark}

\numberwithin{figure}{subsection}
\numberwithin{equation}{section}

\renewcommand{\le}{\leqslant}
\renewcommand{\ge}{\geqslant}
\renewcommand{\leq}{\leqslant}
\renewcommand{\geq}{\geqslant}

\newcommand{\Bell}{\mathbb{B}}
\newcommand{\BellcPlus}{\Bell_{c,+}}
\newcommand{\BellcMinus}{\Bell_{c,-}}

\newcommand{\BellcC}{\Bell_{c,+}^{\mathbb{C}}}
\newcommand{\BelldC}{\Bell_{d,+}^{\mathbb{C}}}
\DeclareMathOperator{\BMO}{BMO}
\newcommand{\Omegac}{\Omega_c}
\newcommand{\OmegacC}{\Omega_{c,\mathbb{C}}}
\newcommand{\OmegadC}{\Omega_{d,\mathbb{C}}}
\newcommand{\Omegad}{\Omega_d}
\newcommand{\BelldPlus}{\Bell_{d,+}}
\newcommand{\BelldMinus}{\Bell_{d,-}}

\newcommand{\av}[2]{\langle {#1}\rangle_{{}_{#2}}}
\DeclareMathOperator{\conv}{conv}
\DeclareMathOperator{\sign}{sign}
\newcommand{\cii}[1]{_{{}_{#1}}}
\newcommand{\df}{\buildrel\mathrm{def}\over=}
\newcommand{\eps}{\varepsilon}
\newcommand{\p}{\theta}

\newcommand{\FixedBoundary}{\partial_{\mathrm{skel}}}

\newcommand{\diff}{\mathrm{d}}

\newcommand{\Pop}{{\mathbf P}}

\newcommand{\Nop}{\mathbf{N}}

\DeclareMathOperator{\re}{\mathrm{Re}}

\renewcommand{\ge}{\geqslant}
\renewcommand{\leq}{\leqslant}
\renewcommand{\geq}{\geqslant}

\newcommand{\Ordo}{\mathrm{O}}

\renewcommand{\phi}{\varphi}

\newcommand{\pd}[2]{\frac{\partial #1}{\partial #2}}

\let\comma=,
{\catcode`\,=\active \gdef,{\relax\ifmmode\comma\else{\upshape\comma}\fi}

\begin{document}

\title{Sharpening H\"older's inequality}

\author{H. Hedenmalm\thanks{Support by RSF 
grant 14-41-00010} \and D. M. Stolyarov
\thanks{Support by RSF grant 14-21-00035}
 \and V. I. Vasyunin\thanksmark{1} 
\and P. B. Zatitskiy\thanksmark{2} 
}


%




\maketitle

\begin{abstract} 
We strengthen H\"older's inequality.  The new family of sharp inequalities 
we obtain might be thought of as an analog of Pythagorean theorem for 
the~$L^p$-spaces. Our treatment of the subject matter is based on 
Bellman functions of four variables. 
\end{abstract}

\section{Introduction} 

\subsection{The Cauchy-Schwarz inequality and the 
Pythagorean theorem}
\label{subsec-1.1}

Let~$\mathcal{H}$ be a Hilbert space (over the complex or the reals) with 
an inner product~$\langle\cdot,\cdot\rangle$. The Pythagorean theorem asserts
\begin{equation}
\Big|\Big\langle f,\frac{e}{\|e\|}\Big\rangle\Big|^2
+\|\Pop_{e^\perp}f\|^2=\|f\|^2,\qquad e,f\in\mathcal{H},\,\,\,e\ne0. 
\label{eq-preBessel1}
\end{equation}
Here,~$\Pop_{e^\perp}$ denotes the orthogonal projection onto the orthogonal 
complement of a nontrivial vector~$e$. 
At this point, we note that since~$\|\Pop_{e^\perp}f\|\ge0$, 
the identity~\eqref{eq-preBessel1} implies the 
Cauchy-Schwarz inequality 
\begin{equation*}
|\langle f,e\rangle|\le \|f\|\,\|e\|,\qquad e,f\in\mathcal{H}.
\end{equation*}
We also note that~\eqref{eq-preBessel1} leads to Bessel's inequality:
\begin{equation*}
\sum_{n=1}^{N}|\langle f,e_n\rangle|^2\le \|f\|^2,\qquad f\in\mathcal{H},
\end{equation*}
for an orthonormal system~$e_1,\ldots,e_N$ in~$\mathcal{H}$.

We may think of~\eqref{eq-preBessel1} as of an expression of the precise 
loss in the Cauchy-Schwarz inequality. Our aim in this paper is to find an 
analogous improvement for the well-known H\"older inequality for~$L^p$ norms.
Before we turn to the analysis of~$L^p$ spaces, we need to replace the 
norm of the projection,~$\|\Pop_{e^\perp}f\|$, by an expression which does not
rely on the Hilbert space structure. It is well known that
\begin{equation}
\|\Pop_{e^\perp}f\|=\inf_\alpha\|f-\alpha e\|,
\label{eq-preBessel2}
\end{equation}
where $\alpha$ ranges over all scalars (real or complex). 

\subsection{
Background on H\"older's inequality for~$L^\theta$}
\label{subsec-1.2}

We now consider~$L^\theta(X,\mu)$, where~$(X,\mu)$ is a standard 
$\sigma$-finite measure space. We sometimes focus our attention on finite 
measures, but typically the transfer to the more general $\sigma$-finite case 
is an easy exercise. The functions are assumed complex valued. 
Throughout the paper 
we assume the summability exponents are in the interval $(1,+\infty)$, 
in particular, $1<\theta<+\infty$. 
We reserve the symbol~$p$ for the range~$[2,\infty)$ and~$q$ for~$(1,2]$ 
(we also usually assume that~$p$ and~$q$ are dual in the sense
~$\frac1p + \frac1q = 1$).  
Also, to simplify the presentation, we assume~$\mu$ has no atoms.  

Our point of departure is H\"older's inequality, which asserts that in terms 
of the sesquilinear form
\begin{equation*}
\langle f,g\rangle_\mu:=\int_X f\bar g\,\diff\mu,
\end{equation*} 
we have 
\begin{equation}
|\langle f,g\rangle_\mu|\le \|f\|_{L^\theta(\mu)}\|g\|_{L^{\theta'}(\mu)},\qquad 
f\in L^\theta(X,\mu),\,\,\,g\in L^{\theta'}(X,\mu),
\quad\frac{1}{\theta}+\frac{1}{\theta'}=1,
\label{eq-Hold1}
\end{equation}
so that~$\theta'=\theta/(\theta-1)$ is the dual exponent. H\"older's 
inequality was found independently by Rogers~\cite{Rogers} and 
H\"older~\cite{Holder}.
It is well-known that, for non-zero functions, equality occurs in 
H\"older's inequality~\eqref{eq-Hold1}
if and only if~$f$ has the form~$f=\alpha \Nop_{\theta'}(g)$ for a 
scalar~$\alpha\in\mathbb{C}$. Here,~$\Nop_r$ denotes the nonlinear operator
\begin{equation*}
\Nop_r(h)(x)=\begin{cases}|h(x)|^{r-2}h(x),\quad &h(x)\ne 0;\\
0,\quad &h(x) = 0,
\end{cases} \qquad r \in (1,\infty).
\end{equation*}
Such operators appear naturally in the context of generalized orthogonality 
for $L^p$ spaces (see, e.g., Chapter 4 of Shapiro's book \cite{Shapbook}). 
We note that~$\Nop_\theta$ and~$\Nop_{\theta'}$ are each other's inverses, 
since~$\Nop_\theta(\Nop_{\theta'}(h))=h$ 
and~$\Nop_{\theta'}(\Nop_{\theta}(h))=h$ for
an arbitrary function~$h$. In addition,~$\Nop_{\theta'}$ 
maps~$L^{\theta'}(X,\mu)$ to~$L^{\theta}(X,\mu)$ 
with good control of norms:
\begin{equation*}
\|\Nop_{\theta'}(h)\|_{L^{\theta}(\mu)}^{\theta}=\int_X|h|^{\theta(\theta'-1)}\,\diff\mu=
\int_X|h|^{\theta'}\,\diff\mu=\|h\|_{L^{\theta'}(\mu)}^{\theta'},
\qquad h\in L^{\theta'}(X,\mu).
\end{equation*} 

\subsection{Possible improvement of H\"older's inequality}
We rewrite H\"older's inequality~\eqref{eq-Hold1} in the form
\begin{equation}
\bigg|\bigg\langle f,\frac{g}{\|g\|_{L^{\theta'}(\mu)}}\bigg\rangle_\mu\bigg|^{r}
\le\|f\|^r_{L^\theta(\mu)},\qquad 
f\in L^{\theta}(X,\mu),\,\,\,g\in L^{\theta'}(X,\mu),\,\,\,g\ne0,
\label{eq-Hold2}
\end{equation}
where~$r$ is real and positive. 
The natural choices for~$r$ are~$r=\theta$ and~$r=\theta'$. This looks a 
lot like the
Pythagorean theorem, only that the projection term is missing. 
Indeed, if~$\theta=\theta'=r=2$, the inequality~\eqref{eq-Hold2} expresses 
exactly Pythagorean theorem~\eqref{eq-preBessel1} 
with the projection term suppressed (with~$e=g$). 
How can we find a replacement of the projection term for 
arbitrary~$\theta \ne 2$?
The key to this lies in the already observed fact that we have equality 
in~\eqref{eq-Hold2} 
if and only if~$f=\alpha \Nop_{\theta'}(g)$ for a scalar~$\alpha\in\mathbb{C}$. 
To see things more clearly, let us agree to 
write~$e =\Nop_{\theta'}(g)\in L^p(X,\mu)$ 
and insert this into~\eqref{eq-Hold2}:
\begin{equation*}
\bigg|\bigg\langle f,\frac{\Nop_\theta(e)}{\|e\|_{L^{\theta}(\mu)}^{\theta-1}}
\bigg\rangle_\mu\bigg|^{r}
\le\|f\|^r_{L^\theta(\mu)},\qquad 
e,f\in L^\theta(X,\mu),\,\,\,e\ne0.
\end{equation*}
Now, looking at~\eqref{eq-preBessel2}, knowing that equality in the 
previous inequality 
holds only when~$f$ is a scalar multiple of~$e$, we posit 
the inequality
\begin{equation}
\bigg|\bigg\langle f,\frac{\Nop_\theta(e)}{\|e\|_{L^{\theta}(\mu)}^{\theta-1}}
\bigg\rangle_\mu\bigg|^{r}+c_{\theta,r}\inf_\alpha\|f-\alpha e\|_{L^\theta(\mu)}^r
\le\|f\|^r_{L^\theta(\mu)},\qquad 
e,f\in L^\theta(X,\mu),\,\,\,e\ne0,
\label{eq-Hold3}
\end{equation}
for some constant~$c_{\theta,r}$, $0\le c_{\theta,r}\le1$ (this constant does 
not depend on~$f$ or~$e$, it depends on~$\theta$ and~$r$ only). 
The inequality~\eqref{eq-Hold3}
cannot hold for any constant~$c_{\theta,r}>1$. Indeed, to see this, it suffices
to pick nontrivial~$e$ and~$f$ such that 
\begin{equation*}
\inf_\alpha\|f-\alpha e\|^r_{L^\theta(\mu)}=\|f\|^r_{L^\theta(\mu)},
\end{equation*}
which means that the minimum is attained at~$\alpha=0$. This is easy to do 
for any given~$e$ by simply replacing the function~$f$ by~$f-\alpha^{\star} e$ 
(here,~$\alpha=\alpha^{\star}$ is a point where the infimum is attained).
The inequality~\eqref{eq-Hold3} is appropriate for iteration, using a 
sequence of functions~$e_1,e_2,e_3,\ldots$, 
as in the case of Bessel's inequality, but if~$c_{\theta,r}<1$
there is an exponential decay of the coefficients in the analogue of Bessel's
inequality. 
We note that the inequality~\eqref{eq-Hold3} holds trivially for~$c_{\theta,r}=0$ 
by H\"older's inequality, and gets stronger the bigger~$c_{\theta,r}$ 
is allowed to be. 
\begin{defn}
Let~$c_{\theta,r}^\star$ denote the the largest possible 
value of~$c_{\theta,r}$ such that~\eqref{eq-Hold3} remains valid for any 
complex-valued~$f$ and~$e$.
\end{defn} 
Clearly,~$0\le c_{\theta,r}^\star\le1$. 
As we will see later,~$c_{\theta,r}^{\star} < 1$ unless~$\theta = 2$.
We slightly transfer the nonlinearity in~\eqref{eq-Hold3} from~$e$ to~$f$, and
posit the inequality
\begin{equation}
\bigg|\bigg\langle \Nop_\theta(f),\frac{e}{\|e\|_{L^{\theta}(\mu)}}
\bigg\rangle_\mu\bigg|^{\frac{r}{\theta-1}}+d_{\theta,r}\inf_\alpha
\|f-\alpha e\|_{L^{\theta}(\mu)}^{r}
\le\|f\|^r_{L^\theta(\mu)},\qquad 
e,f\in L^\theta(X,\mu),\,\,\,e\ne0,
\label{eq-Hold4}
\end{equation}
where~$d_{\theta,r}\ge0$ is real.
If we again argue that we can find nontrivial functions~$e$ and~$f$ such 
that the infimum in~\eqref{eq-Hold4} is attained at~$\alpha=\alpha^{\star}=0$, 
then it is immediate that~\eqref{eq-Hold4} cannot be valid generally 
unless~$d_{\theta,r}\le1$, that is,~$d_{\theta,r}$ must be 
confined to~$0\le d_{\theta,r}\le1$.
\begin{defn}
Let~$d_{\theta,r}^\star$ denote the the largest possible 
value of~$d_{\theta,r}$ such that~\eqref{eq-Hold4} remains valid for 
any complex-valued~$f$ and~$e$.
\end{defn} 

\begin{rem}
The constants~$c_{\theta,r}^{\star}$ and~$d_{\theta,r}^{\star}$ do not depend 
on the measure space $(X,\mu)$ as long as the measure $\mu$ is continuous.
There are essentially two cases, the finite mass case and the infinite mass
$\sigma$-finite case. By normalization, the finite mass case becomes the
probability measure case, 
and all the standard probability measure spaces without point masses are 
isomorphic. The infinite mass $\sigma$-finite case is then treated as a the 
limit of the finite mass case. 
We explain some details later on in the proof of 
Proposition~\ref{CFromBell} at the end of Section~\ref{s2}. 
Although this is deferred until later, we will need it in what follows.
\end{rem}

In certain ranges of~$r$, we cannot get more than H\"older's inequality. 
We describe these restrictions in two lemmas below. 

\begin{lem}\label{rlessthanp}
If~$r < \theta$, then~$c_{\theta,r}^{\star} = d_{\theta,r}^{\star} = 0$.
\end{lem}
\begin{proof}
We consider the finite mass case, and renormalize to have mass $1$. 
For the $\sigma$-finite case, we would instead just cut off a piece of our 
measure space of mass $1$ and consider functions that vanish off that
piece. In the mass $1$ case, we apply standard probability measure space 
theory, and take $X=[0,1]$, with $\mu$ as Lebesgue measure.
Moreover, we let $e$ be the constant function $e=1$ while $f$ is given by 
the formula
\begin{equation*}
f(x) = \begin{cases}
2,\quad &x \in [0,\eps);\\
0,\quad &x \in [\eps,2\eps);\\
1,\quad &x \in [2\eps,1].
\end{cases}
\end{equation*}
Here~$\eps$ is a real parameter with $0<\epsilon<\frac12$. 
By symmetry and convexity with respect to~$\alpha$,
\begin{equation*}
\inf_\alpha \|f-e-\alpha e\|_{L^\theta}^r = 
\|f-e\|_{L^\theta}^r = (2\eps)^{\frac{r}{\theta}}.
\end{equation*}
Moreover, we observe that $\langle f,e \rangle = 1$ while
$\|f\|_{L^\theta}^\theta =\eps 2^\theta + (1-2\eps)$. 
Thus,~\eqref{eq-Hold3} leads to
\begin{equation*}
1+c_{\theta,r}(2\eps)^{\frac{r}{\theta}} \leq \Big(\eps 2^\theta 
+(1-2\eps)\Big)^{\frac{r}{\theta}}.
\end{equation*}
We subtract 1, divide by~$\eps^{\frac{r}{\theta}}$, compute the limit 
as~$\eps \to 0$ and obtain $0$ on the right-hand side if~$r < \theta$. 
The conclusion that the optimal constant is $c_{\theta,r}^\star=0$ is immediate
from this. 
The same choice of~$f$ and~$e$ also gives that $d_{\theta,r}^{\star} = 0$.
\end{proof}

\begin{lem}\label{rlessthan2}
If~$r < 2$, then~$c_{\theta,r}^{\star} = d_{\theta,r}^{\star} = 0$.
\end{lem}
\begin{proof}
We prove the lemma for the case of constants~$d_{\theta,r}^{\star}$ 
this time. We take as before $X=[0,1]$ and $\mu$ as Lebesgue measure.
We consider the functions $e=1$ and~$f=1+th$, where~$h$ is a bounded 
real-valued function 
with symmetric distribution (by which we mean that the functions~$h$ and~$-h$ 
have one and the same distribution) and~$t$ is a real parameter which will tend 
to zero. By the symmetry of~$h$,
\begin{equation*}
\inf_\alpha
\|f-\alpha e\|_{L^{\theta}}^{r} = \inf_{\beta}\|th - \beta e\|_{L^\theta}^r =\|th
\|_{L^\theta}^r=t^{r}\|h\|_{L^\theta}^r.
\end{equation*}
Again by the symmetry of $h$ it follows that $\langle h,e\rangle=0$, 
and we obtain as $t\to0$ that
\begin{equation*}
|\langle \Nop_\theta(f),e\rangle|^{\frac{r}{\theta-1}} = 
\Big(\langle (1+th)^{\theta-1},e\rangle\Big)^{\frac{r}{\theta-1}} = 
\bigg(1+t(\theta-1)\langle h,e\rangle + \Ordo(t^2)\bigg)^{\frac{r}{\theta-1}} 
= 1+ \Ordo(t^2)
\end{equation*}
and, similarly,
\begin{equation*}
\|f\|_{L^\theta}^r = 1+\Ordo(t^2)
\end{equation*}
as~$t$ tends to zero.
By plugging these asymptotic identities back into~\eqref{eq-Hold4}, we arrive 
at
\begin{equation*}
d_{\theta,r}^{\star}t^{r}\|h\|_{L^\theta}^r = \Ordo(t^2),
\end{equation*}
which either proves that~$r \geq 2$ or that $d_{\theta,r}^{\star}= 0$. The case 
with the constants~$c_{\theta,r}^{\star}$ is similar.
\end{proof}
 
The inequalities~\eqref{eq-Hold3} and~\eqref{eq-Hold4} tend to get sharper 
the smaller~$r$ is. Indeed, since for fixed $\gamma$ with $0\le\gamma\le1$, 
the inequality
\[
A^r+\gamma B^r\le1\quad\text{with}\quad 0\le A,B\le1, 
\]
for a fixed positive $r=r_0$ implies the same inequality for all 
$r\ge r_0$, it follows that for fixed $\theta$, the functions 
$r\mapsto c^\star_{\theta,r}$ and $r\mapsto d^\star_{\theta,r}$ grow with $r$, 
and, moreover, if one of these constants already equals $1$, then in 
\eqref{eq-Hold3} or alternatively \eqref{eq-Hold4} we should use the 
smallest possible $r$ so that this
remains true because that represents the strongest assertion.
Our Lemmas~\ref{rlessthanp} and~\ref{rlessthan2} suggest 
that the cases~$r = 2$ for~$\theta \leq 2$ and~$r = \theta$ 
for~$\theta \geq 2$ might be the most interesting. 
It appears that one may compute the constants~$c_{p,p}^{\star}$ 
and~$d_{p,p}^{\star}$ for~$p > 2$. Here are our two main results.

\begin{thm}\label{Cthm}
Suppose~$2<p<+\infty$. Then the optimal constant~$c_{p,p}^{\star}$ in the 
inequality~\eqref{eq-Hold3} may be computed as
\begin{equation*}
c_{p,p}^{\star} = (p-1)\Big(\frac{s_0}{1+s_0}\Big)^{p-2},
\end{equation*}
where~$s_0$ is the unique positive solution of the equation
\begin{equation}
\label{eqs0}
(p-1)s_0^{p-2} + (p-2)s_0^{p-1}=1.
\end{equation}

\end{thm}
Using Taylor's formula, we may find the asymptotic expansion 
for~$c_{p,p}^{\star}$ as~$p\to 2+$:
\begin{equation*}
c_{p,p}^{\star} = 1-(p-2)\Big(-1+\log\frac{1+w}{w}\Big)+O\big((p-2)^2\big),
\end{equation*}
where~$w = W(\frac{1}{e})$. Here, $W$ denotes the Lambert-$W$ function, i.e. 
the solution of the equation~$W(z)e^{W(z)}=z$.

\begin{thm}\label{Dthm}
Let~$2<p<+\infty$. Then the optimal constant~$d_{p,p}^{\star}$ is given by 
the formula
\begin{equation*}
d_{p,p}^{\star} = (q-1)c_{p,p}^{\star}= \Big(\frac{s_0}{1+s_0}\Big)^{p-2},
\end{equation*}
where~$s_0$ is given by~\eqref{eqs0} and $q$ is dual to $p$, that is, 
$\frac{1}{p}+\frac{1}{q}=1$.
\end{thm}

The case of exponent $\theta<2$ appears to be somewhat more elementary. 
At least the theorem stated below is considerably easier to obtain than the two 
theorems above.

\begin{thm}\label{Elementary}
The optimal constant~$d_{q,r}^{\star}$ is given by the formula
\begin{equation*}
d_{q,r}^{\star} = 1
\end{equation*}
when~$1<q\le2$ and $2\le r<+\infty$. 
Moreover, $d_{p,r}^{\star}=1$ holds in the range $2<p<+\infty$ if and only if
$r \geq 2(p-1)$.
\end{thm}

Our method allows us to compute the constants~$c_{\theta,r}^{\star}$ 
and~$d_{\theta,r}^{\star}$ for the case of arbitrary~$r$ and~$p$. However, 
the answer does not seem to have a short formulation. At least, we 
provide sharp constants for all endpoint cases, and also indicate the 
domain where~$d_{p,r}^{\star} = 1$. Figure~\ref{fig:cprdpr} shows two diagrams
which illustrate what we know about the optimal constants $c_{\theta,r}^{\star}$ 
and~$d_{\theta,r}^{\star}$.

\begin{center}
\begin{figure}
\includegraphics[width=7cm]{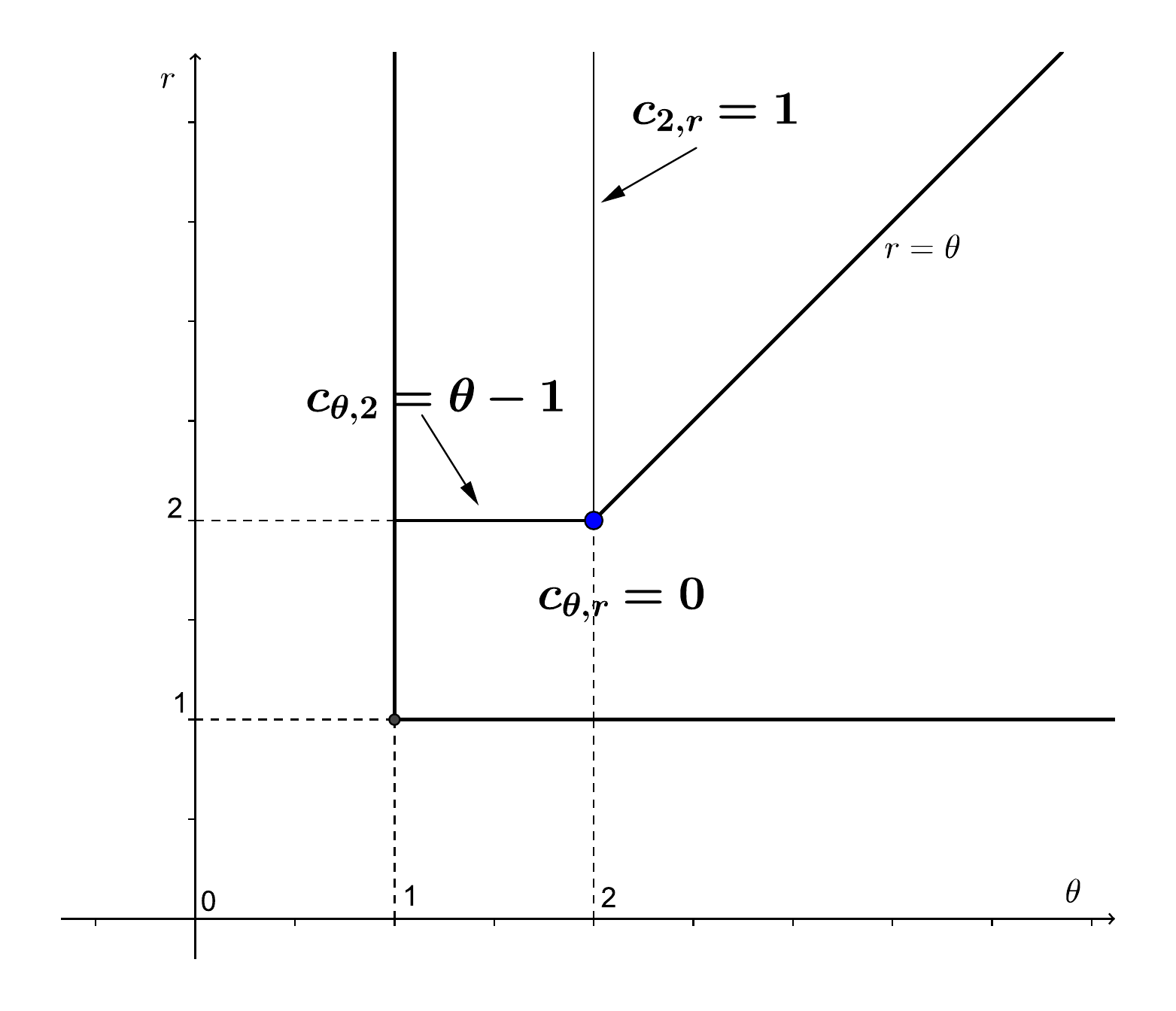}
\includegraphics[width=7cm]{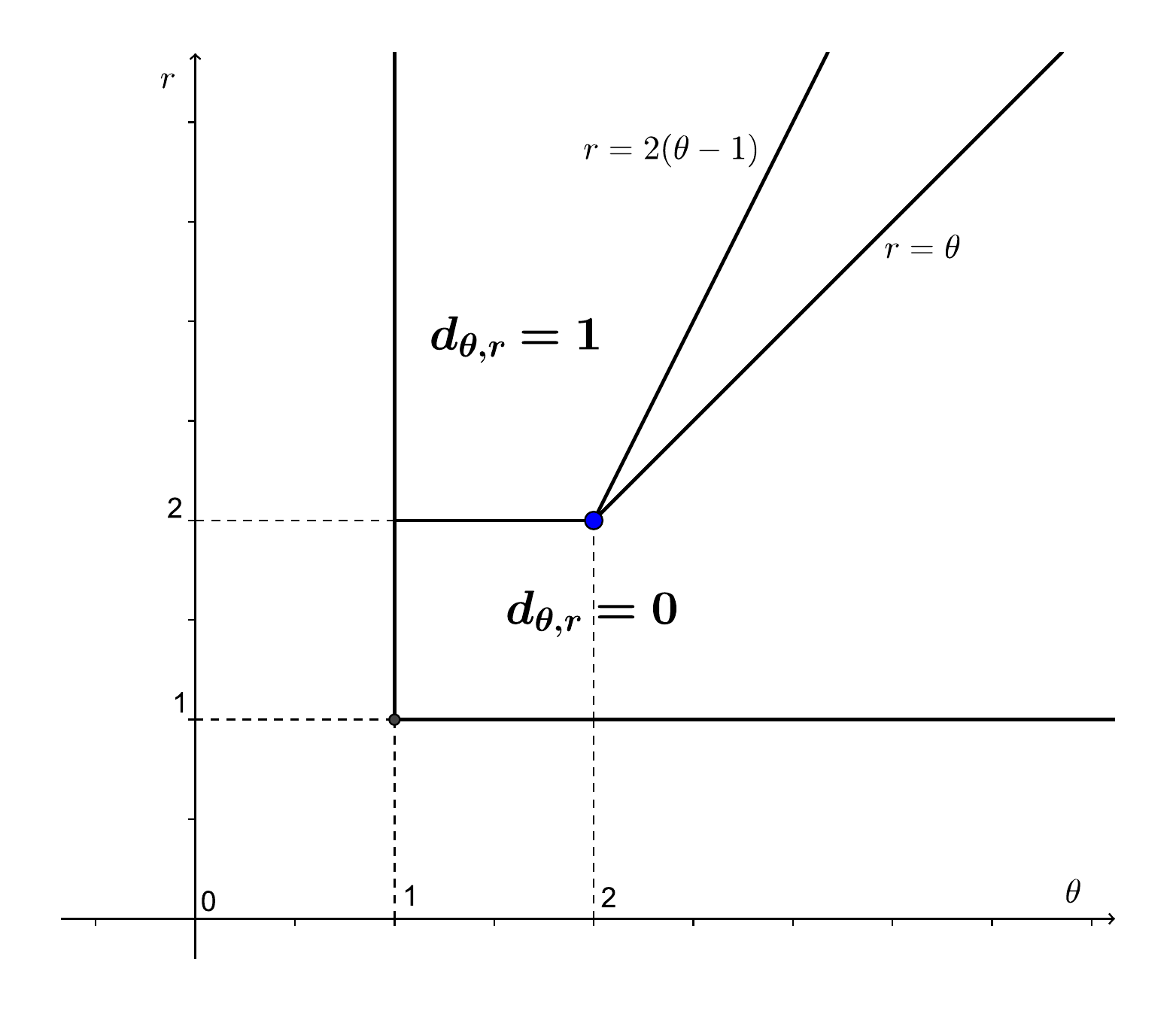}
\caption{Values of $c_{\theta,r}^{\star}$ and $d_{\theta,r}^{\star}$.}
\label{fig:cprdpr}
\end{figure}
\end{center}
 
\begin{rem}
(a) Our results sharpen H\"older's inequality. 
There is a constant interest in sharper forms of the classical inequalities, 
for which the optimizers have already been described. Such sharpenings 
may be viewed as \emph{stability results}: the new inequality says that if the 
equality almost holds, then the functions are close to the optimizers. 
See~\cite{Christ1} for the Hausdorff-Young and Young's convolutional 
inequalities,~\cite{Christ2} for the Riesz-Sobolev inequality,~\cite{BO} 
for various martingale inequalities, and~\cite{Carlen} for Sobolev-type 
embedding theorems. The latter paper also suggests a theoretical approach 
to the stability  phenomenon. The list of references is far from being 
complete.

\noindent (b) The stability of Minkowski's inequality (which is closely 
associated with H\"older's inequality)  is related to the 
notion of uniform convexity of Lebesgue spaces introduced by 
Clarkson~\cite{Clarkson}. 
He also found the sharp estimates for modulus of convexity of~$L^{\theta}$ 
in the case~$\theta \geq 2$. The sharp form of the uniform convexity 
inequalities for~$\theta < 2$ was given by Hanner's inequalities which were
first obtained by Beurling, see the classical paper of Hanner~\cite{Hanner} 
as well as a more modern exposition~\cite{ISZ}. The technique of 
the latter paper is very close to what we are using in our work. We should 
also mention that it is well-known that the notion of uniform convexity 
may be extended to Schatten classes in place of the $L^p$ spaces, and 
moreover, sharp results such as Hanner's inequality might be available in 
this more general setting (see~\cite{BCL}). It is interesting to ask whether 
something of this sort is available also for the sharper forms of H\"older's 
inequalities.

\noindent (c) In some applications of H\"older's inequality, the instance
of exponent $p=2$ might not be applicable but for instance $p>2$ close to $2$
is. A case in point is the paper by Baranov and Hedenmalm \cite{BarHed}.
It would be of interest to see what the sharpened forms derived here will 
be able to lead to in terms of strengthened results in that context. 
\end{rem}

\subsection{Acknowledgement}
Theorem 1.6 and the material of Sections 2-3 is supported by Russian
Science Foundation grant 14-21-00035, while Theorems 1.7, 1.8 and the 
material of Sections 4-9 is supported by Russian Science Foundation 
grant 14-41-00010. 

\subsection{Bellman function}
\label{s14}
For a measurable subset~$E$ of an interval~$I$ and a summable f
unction~$f\colon I \to \mathbb{C}$, we denote the average of~$f$ 
over~$E$ by~$\av{f}{E} = |E|^{-1}\int_E f(s)\,\diff m(s)$. The symbol~$m$ 
denotes the Lebesgue measure and~$|E| = m(E)$ by definition.

\begin{defn}
The constant~$c_{\theta,r}^{\star,\mathbb{R}}$ is the largest possible 
constant~$c_{\theta,r}$ such that the inequality~\eqref{eq-Hold3} holds 
true for all real-valued functions~$f$ and~$e$. The 
constant~$d_{\theta,r}^{\star,\mathbb{R}}$ is the largest possible 
constant~$d_{\theta,r}$ such that the inequality~\eqref{eq-Hold4} holds true for all real-valued functions~$f$ and~$e$.
\end{defn}
We will express the constants~$c_{\theta,r}^{\star,\mathbb{R}}$ 
and~$d_{\theta,r}^{\star,\mathbb{R}}$ in terms of two Bellman functions, 
which are very similar. We introduce the main one.

\begin{defn}\label{Bellman}
Let~$I\subset \mathbb{R}$ be a finite interval. 
Consider the function~$\BellcPlus\colon \mathbb{R}^4\to \mathbb{R}$,
\begin{equation*}
\BellcPlus(x_1,x_2,x_3,x_4) = \sup\Big\{\av{fg}{I}\,\Big|\;f,g 
\hbox{ real-valued, } \av{f}{I}=x_1, 
\av{g}{I}=x_2, \av{|f|^{\theta}}{I}=x_3, \av{|g|^{\theta'}}{I}=x_4\Big\}.
\end{equation*}
\end{defn}

\begin{prop}\label{CFromBell}
For any~$\theta$ and any~$r$, we have that
\begin{equation*}
c_{\theta,r}^{\star,\mathbb{R}} = \bigg(\sup_{x_1 \in (-1,1)} 
\frac{\BellcPlus^r(x_1,0,1,1)}{1 - |x_1|^{r}}\bigg)^{-1}.
\end{equation*}
\end{prop}

\begin{defn}
Let~$I\subset \mathbb{R}$ be a finite interval. 
Consider the function~$\BelldPlus\colon \mathbb{R}^4\to \mathbb{R}$,
\begin{equation*}
\BelldPlus(x_1,x_3,x_4,x_5) = \sup\Big\{\av{g}{I}\,\Big|\;f,g 
\hbox{ real-valued, }  \av{f}{I}=x_1, \av{|f|^{\theta}}{I}=x_3, 
\av{|g|^{\theta'}}{I}=x_4, 
\av{fg}{I}=x_5\Big\}.
\end{equation*}
\end{defn}

\begin{rem}
The function~$\BelldPlus$ depends on four variables~$x_1,x_3,x_4,x_5$. 
Though such a choice of variables might seem strange, it will appear to 
be very natural later. In particular, it makes the link 
between~$\BelldPlus$ and~$\BellcPlus$ more transparent, see 
Corollary~\ref{GeneralConvex2} below.
\end{rem}

\begin{prop}\label{DFromBell}
For any~$\theta$ and~$r$,
\begin{equation*}
d_{\theta,r}^{\star,\mathbb{R}} = \bigg(\sup_{x_1 \in (-1,1)} 
\frac{\BelldPlus^r(x_1,1,1,0)}{1 - |x_1|^{\frac{\theta'r}{\theta}}}\bigg)^{-1}.
\end{equation*}
\end{prop}
We will compute the functions~$\BellcPlus$ and~$\BelldPlus$ 
for~$\theta=p \geq 2$. The answer is rather complicated, so we will 
state it slightly later (see Theorems~\ref{BellmanTheorem} 
and~\ref{BellmanTheorem2} below). These 
functions are solutions of specific minimization problems on 
subdomains of~$\mathbb{R}^4$, they allow geometric interpretation.

The computation of the constants~$c_{p,p}^{\star,\mathbb{R}}$ 
and~$d_{p,p}^{\star,\mathbb{R}}$ leads to the proof of 
Theorems~\ref{Cthm} and~\ref{Dthm} via the propositions below.
\begin{prop}\label{Complexc}
For any~$\p \in (1,\infty)$ and any~$r$ we 
have~$c_{\p,r}^{\star,\mathbb{R}} = c_{\p,r}^{\star}$.
\end{prop}
\begin{prop}\label{Complexd}
For any~$\p \in (1,\infty)$ and any~$r$ we 
have~$d_{\p,r}^{\star,\mathbb{R}} = d_{\p,r}^{\star}$.
\end{prop}

Though the present 
paper is self-contained, it employs the heuristic experience 
acquired by the authors as a result of study
of other problems. The closest one is the Bellman function 
in~\cite{ISZ}. See~\cite{ShortReport2} for a
more geometric point of view and~\cite{ShortReport},~\cite{TAMS}, 
and~\cite{Mem} for a study of a related problem.
We also refer the reader 
to~\cite{NTV},~\cite{Osekowski},~\cite{SV},~and~\cite{Volberg} for 
history and basics of the Bellman function theory.
It would appear that all the previously computed sharp Bellman functions were 
either two or three dimensional. Our functions~$\BellcPlus$
and~$\BelldPlus$ depend on four variables.

\paragraph{Organization of the paper.} In Section~\ref{s2}, we 
study simple properties of the functions~$\BellcPlus$ and~$\BelldPlus$ 
and prove our Propositions~\ref{CFromBell} and~\ref{DFromBell}. We 
also introduce more Bellman functions here. In Section~\ref{s3}, 
we compute~$\BellcPlus$ and describe its foliation.
Section~\ref{s4} contains the computation of the 
constants~$c_{p,r}^{\star,\mathbb{R}}$. In Section~\ref{s5},
we prove Proposition~\ref{Complexc}, in particular, we  prove 
Theorem~\ref{Cthm}. Then, in Section~\ref{s6},
we find~$\BelldPlus$. In fact, it can be expressed in terms 
of~$\BellcPlus$ and its companion minimal function by a change of variables. 
Next, we compute~$d_{p,r}^{\star,\mathbb{R}}$ in Section~\ref{s7} and 
prove Proposition~\ref{Complexd} in Section~\ref{s8}. We finish the 
paper with the elementary proof of Theorem~\ref{Elementary} given in 
Section~\ref{s9}.

\section{Basic properties of Bellman functions}\label{s2}
We assume~$p \geq 2$. By H\"older's inequality, there does not 
exist~$f$ or~$g$ such 
that~$\av{|f|^p}{I} < |\av{f}{I}|^p$ or~$\av{|g|^q}{I} < |\av{g}{I}|^q$. 
Therefore, the function~$\BellcPlus$ is equal to~$-\infty$ outside the set
\begin{equation*}
\Omegac = \Big\{x\in \mathbb{R}^4\,\Big|\; x_3 \geq |x_1|^p, 
x_4 \geq |x_2|^q\Big\}.
\end{equation*}
On the other hand, since for any~$x \in \Omegac$ there exist 
functions~$f$ and~$g$ such that~$x_1 = \av{f}{I}$,~$x_2 = 
\av{g}{I}$,~$x_3 = \av{|f|^p}{I}$, and~$x_4 = \av{|g|^q}{I}$, 
we have~$\BellcPlus(x) > -\infty$ for~$x \in \Omegac$. We call~$\Omegac$ 
the natural domain (or simply the domain) of~$\BellcPlus$. The set
\begin{equation*}
\FixedBoundary\Omegac = \Big\{x\in \mathbb{R}^4\,\Big|\; 
x_3 = |x_1|^p, x_4 = |x_2|^q\Big\}
\end{equation*}
is called the skeleton of~$\Omegac$ (this is the set of the extreme 
points of~$\Omegac$; note that it is only a small part of the topological 
boundary). Similarly,
\begin{equation*}
\Omegad=\Big\{x\in \mathbb{R}^4\,\Big|\; x_4\geq 0, \, |x_1| \leq x_3^{1/p}, 
\, |x_5|\leq x_3^{1/p}x_4^{1/q}\Big\}
\end{equation*}
is the domain of~$\BelldPlus$ and
\begin{equation*}
\FixedBoundary\Omegad = \Big\{x\in \mathbb{R}^4\,\Big|\;x_4\geq 0, \, 
|x_1| =  x_3^{1/p}, \, |x_5| = x_3^{1/p}x_4^{1/q}\Big\}
\end{equation*}
is the skeleton of~$\Omegad$.
\begin{lem}\label{SimpleProperties}
The functions~$\BellcPlus$ and~$\BelldPlus$ satisfy the following 
properties\textup:
\begin{itemize}
\item They do not depend on the interval~$I$\textup;
\item They satisfy the boundary conditions
\begin{equation*}
\begin{aligned}
\BellcPlus(x_1,x_2,|x_1|^p,|x_2|^q) &= x_1x_2;\\
\BelldPlus(x_1,|x_1|^{p},|x_2|^{q},x_1x_2) &= x_2;
\end{aligned}
\end{equation*}
\item They are positively homogeneous\textup, for 
positive~$\lambda_1$ and~$\lambda_2$, we have\textup:
\begin{equation*}
\begin{aligned}
\BellcPlus(\lambda_1 x_1,\lambda_2 x_2,\lambda_1^p x_3,\lambda_2^qx_4) 
&= \lambda_1\lambda_2\BellcPlus(x);\\
\BelldPlus(\lambda_1 x_1,\lambda_1^p x_3,\lambda_2^q x_4,
\lambda_1 \lambda_2 x_5) &= \lambda_2 \BelldPlus(x);
\end{aligned}
\end{equation*}
\item They are pointwise minimal among all concave functions 
on their domains that satisfy the same boundary conditions. 
\end{itemize}
\end{lem}
This lemma is very standard, see Propositions $1$ and~$2$ in~\cite{ISZ}, 
where a completely similar statement is proved for another Bellman function.

Mininal concave functions can be described in terms of convex hulls. 
It is convenient to consider companion functions 
for~$\BellcPlus$ and~$\BelldPlus$.
\begin{defn}
Let~$I\subset \mathbb{R}$ be a finite interval,~$L^p=L^p(I,m)$. 
Consider the functions~$\BellcMinus$ and~$\BelldMinus$ given by the formulas
\begin{equation*}
\BellcMinus(x_1,x_2,x_3,x_4) = \inf\Big\{\av{fg}{I}\,\Big|\;f,g 
\hbox{ real-valued, }  \av{f}{I}=x_1, 
\av{g}{I}=x_2, \av{|f|^p}{I}=x_3, \av{|g|^q}{I}=x_4\Big\};
\end{equation*}
\begin{equation*}
\BelldMinus(x_1,x_3,x_4,x_5) = \inf\Big\{\av{g}{I}\,\Big|\;f,g 
\hbox{ real-valued, }  \av{f}{I}=x_1, \av{|f|^p}{I}=x_3, \av{|g|^q}{I}=x_4, 
\av{fg}{I} = x_5\Big\}.
\end{equation*}
\end{defn}
\begin{rem}
The functions~$\BellcMinus$ and~$\BelldMinus$ satisfy the properties 
similar to those listed in Lemma~\ref{SimpleProperties}: the first 
three properties remain the same, and in the last property, one should 
replace ``minimal concave'' with ``maximal convex''.
\end{rem}
\begin{lem}\label{GeneralConvex}
Let~$\omega$ be a closed convex subset of~$\mathbb{R}^n$, 
let~$\partial \omega$ be the set of its extreme points. Consider a 
continuous function~$f \colon \partial \omega \to \mathbb{R}$. 
Let~$\mathcal{B}^+_f$ and~$\mathcal{B}^-_f$ be the minimal convex and 
the maximal concave functions on~$\omega$ that coincide with~$f$ 
on~$\partial\omega$. The intersection of the subgraph 
of~$\mathcal{B}^+_f$ and the epigraph of~$\mathcal{B}^-_f$ coincides 
with the closure of the convex hull \textup(in~$\mathbb{R}^{n+1}$\textup) 
of the graph of~$f$\textup:
\begin{equation*}
\Big\{(x,y)\in \omega \times \mathbb{R}\,\Big|\; \mathcal{B}^-_f(x) 
\leq y \leq \mathcal{B}^+_f(x)\Big\} = \overline{\conv\Big\{(x, f(x)) 
\,\Big|\; x \in \partial \omega\Big\}}.
\end{equation*}
\end{lem}
We will not prove this lemma. It is a slight generalization of 
Proposition~$3$ of~\cite{ISZ} (in~\cite{ISZ}, we worked with strictly 
convex compact domains). It has an immediate corollary that allows to 
express~$\BelldPlus$ and~$\BelldMinus$ in terms 
of~$\BellcPlus$ and~$\BellcMinus$.

\begin{cor}\label{GeneralConvex2}
Let the five dimensional body~$\mathbb{K}$ be the convex hull of 
the two dimensional surface
\begin{equation*}
\Big\{(t_1,t_2,|t_1|^p,|t_2|^q,t_1t_2) \in \mathbb{R}^5\,\Big|\; 
t_1,t_2 \in \mathbb{R}\Big\}.
\end{equation*} 
Then, on one hand,
\begin{equation*}
\mathbb{K} = \Big\{x \in \mathbb{R}^5\,\Big|\; (x_1,x_2,x_3,x_4) 
\in \Omegac,\ \BellcMinus(x_1,x_2,x_3,x_4) \leq x_5 \leq 
\BellcPlus(x_1,x_2,x_3,x_4)\Big\}
\end{equation*}
and on the other hand
\begin{equation*}
\mathbb{K} =\Big\{x \in \mathbb{R}^5\,\Big|\; (x_1,x_3,x_4,x_5) \in 
\Omegad,\ \BelldMinus(x_1,x_3,x_4,x_5) \leq x_2 \leq 
\BelldPlus(x_1,x_3,x_4,x_5)\Big\}.
\end{equation*}
\end{cor}

Corollary~\ref{GeneralConvex2} says that the graphs of the four 
Bellman functions we consider are parts of the boundary of a certain 
convex hull. We invoke the Carath{\'e}odory theorem (see 
Corollary~$1$ in~\cite{ISZ}) to see for each point~$x\in \Omega$ 
there exists not more than five points~$x_j$ in the set~$\FixedBoundary \Omega$
such that~$x\in\conv\{x_j\}_j$ and~$\Bell$ (which is any of our Bellman 
functions) is linear on~$\conv\{x_j\}$ (see Corollary~$1$ 
in~\cite{ISZ})\footnote{This is not quite correct: our domain is 
not compact and we are not allowed to use the Carath{\'e}odory theorem; 
we do not rely on this reasoning formally.}. 
Thus,~$\Omega$ splits into subsets~$\omega$, which might be one, two, 
three, or four dimensional, such that~$\Bell$ is linear on each~$\omega$. 
Such a splitting is called the foliation of~$\Bell$. The function~$\Bell$ 
can be easily computed once we establish its foliation. 

We also state without proof that for any~$\omega$ there exists a unique 
affine function~$\mathrm{D} \Bell (\omega)$ whose graph is the supporting 
plane of the subgraph of~$\Bell$ at each point~$x \in \omega$. In other 
words, the gradient of~$\Bell$ is one and the same for all the points 
inside each subdomain of the foliation. We state this principle without 
proof since, first, it is not needed for the formal proof (however, it 
helps us to find the answer), second, it requires additional smoothness 
assumptions (which our particular problem does satisfy).

Our strategy of the proof will be to try to find the affine 
functions~$\mathrm{D} \Bell$ and then reconstruct the 
function~$\Bell$ from them. In a sense, we find the convex 
conjugate of~$\Bell$ rather than the function itself.

\begin{proof}[Proof of Proposition~\textup{\ref{CFromBell}}.]
We rewrite inequality~\eqref{eq-Hold3} as
\begin{equation*}\label{eq01}
\inf_{\alpha\in\mathbb{R}}\|f-\alpha e\|_{L^{\theta}(\mu)}^r \leq 
c_{\theta,r}^{-1}\Big(\|f\|_{L^{\theta}(\mu)}^r - 
\big|\langle f,e|e|^{\theta-2}\rangle_{\mu}\big|^r \Big), 
\quad \|e\|_{L^{\theta}(\mu)} = 1.
\end{equation*}
By the Hahn-Banach Theorem, the latter inequality is equivalent to
\begin{equation}\label{eq04}
\big|\langle f,g\rangle_{\mu}\big|^r
\leq c_{\theta,r}^{-1}\Big(\|f\|_{L^{\theta}(\mu)}^r - 
\big|\langle f,e |e|^{\theta-2}\rangle_{\mu}\big|^r \Big),
\quad \|e\|_{L^{\theta}(\mu)} = 1, \ \|g\|_{L^{\theta'}(\mu)}=1 \ 
\hbox{and}\ \langle e,g\rangle_{\mu}=0.
\end{equation}
Without loss of generality, we may assume that~$e\ne 0$ almost everywhere. 
Consider the measure~$d\tilde\mu = |e|^{\theta}d\mu$. This is a continuous 
probability measure. We also consider modified 
functions~$\tilde f = \frac{f}{e}\in L^{\theta}(\tilde\mu)$ 
and~$\tilde g = \frac{ge}{|e|^{\theta}}\in L^{\theta'}(\tilde\mu)$. Clearly,
\begin{multline*}
\langle \tilde f, \tilde g\rangle_{\tilde{\mu}} =\langle f, g\rangle_{\mu},
\quad \|\tilde f\|_{L^{\theta}(\tilde \mu)} = \|f\|_{L^{\theta}(\mu)},\quad 
\|\tilde g\|_{L^{\theta'}(\tilde \mu)} = \|g\|_{L^{\theta'}(\mu)}, 
\\ \langle 1,\tilde g\rangle_{\tilde{\mu}}=\langle e,g\rangle_{\mu}, 
\quad \hbox{and}\quad 
\langle \tilde{f},1\rangle_{\tilde{\mu}} = \langle f,e |e|^{\theta-2}\rangle_{\mu}
\end{multline*}
and~\eqref{eq04} turns into
\begin{equation*}
\big|\av{\tilde{f},\tilde{g}}{\tilde{\mu}}\big|^r \leq 
c_{\theta,r}^{-1}\big(\|\tilde{f}\|_{L^{\theta}(\tilde{\mu})}^r - 
\big|\av{\tilde{f}}{\tilde{\mu}}\big|^r\big).
\end{equation*}
It remains to identify the standard probability space~$(X,\tilde\mu)$ 
with~$([0,1],m)$ and use Definition~\ref{Bellman}.
\end{proof}
The proof of Proposition~\ref{DFromBell} is completely similar.

\section{The computation of~$\BellcPlus$}\label{s3}

\subsection{Statement of results}\label{s31}
We will use three auxiliary functions~$\varphi, \lambda$, and~$\rho$. The 
function~$\varphi\colon \mathbb{R} \to \mathbb{R}$ is simple:
\begin{equation*}
\varphi(R) = R|R|^{p-2}.
\end{equation*}
The function~$\lambda\colon\mathbb{R} \to \mathbb{R}$ is given by the formula
\begin{equation}\label{lambdadef}
\lambda(R) = \begin{cases}
\frac{1}{1+R} - \frac{p-1}{1+\varphi(R)}, \quad & R\ne -1;\\
-\frac{p-2}{2},\quad & R=-1.
\end{cases}
\end{equation}
Note that~$\lambda$ is a continuous function. The function~$\rho$ will 
be defined after the following lemma.
\begin{lem}\label{RealR0}
There exists~$R_0\in (0,1)$ such that
\begin{equation}\label{eqR0}
R_0^{\frac{p-1}{2}}+R_0^{-\frac{p-1}{2}}=(p-1)\big(R_0^{\frac{1}{2}}+R_0^{-\frac{1}{2}}\big).
\end{equation}
The function~$\lambda$ is decreasing on~$(-1,R_0)$ and increasing on~$(R_0,1)$.
\end{lem}
We will prove this technical lemma at the end of the subsection.
\begin{defn}\label{Rho}
Define the function~$\rho\colon [-1,1]\to [R_0,1]$ by the 
formula~$\lambda(\rho(R)) = \lambda(R)$ 
when~$R \in [-1,R_0]$ and~$\rho(R) = R$ when~$R \in [R_0,1]$. 
\end{defn}
Note that~$\lambda(-1) = \lambda(1)$. The function~$\rho$ first decreases 
from~$1$ downto~$R_0$ on~$[-1,R_0]$ and then increases back to~$1$ on~$[R_0,1]$.
\begin{thm}\label{BellmanTheorem}
For any~$R \in [-1,1]$,~$a_1,a_2 \in \mathbb{R}$ such that~$a_1a_2 > 0$, 
the function~$\BellcPlus$ is linear on the segment~$\ell_c(a_1,a_2,R)$ 
with the endpoints
\begin{equation}\label{corabpoints}
a=(a_1,a_2,|a_1|^p,|a_2|^q), 
\qquad b=(-\rho(R)a_1,-\phi(R)a_2,|\rho(R)a_1|^p,|R|^p|a_2|^q).
\end{equation} 
In other words,
\begin{equation}\label{Linearity}
\BellcPlus(\tau a +(1-\tau)b) = a_1a_2(\tau+(1-\tau)\rho(R)\phi(R)), 
\quad \tau \in [0,1].
\end{equation}
If~$R\in [R_0,1]$, then~$\BellcPlus(x)=x_3^{\frac{1}{p}}x_4^{\frac{1}{q}}$ 
on~$\ell_c(a_1,a_2,R)$. The segments~$\ell_c(a_1,a_2,R)$ with~$a_1 > 0$, 
$a_2 > 0$, and~$R\in [-1,1]$ cover the domain
\begin{equation*}
\Big\{x\in \Omega \,\Big|\; -1 < \frac{x_2}{x_4^{\frac{1}{q}}} 
\leq \frac{x_1}{x_3^{\frac{1}{p}}} < 1\Big\}
\end{equation*}
entirely. The segments~$\ell_c(a_1,a_2,R)$ with~$a_1 < 0$,~$a_2 < 0$, 
and~$R \in [-1,1]$ cover the domain
\begin{equation*}
\Big\{x\in \Omega \,\Big|\; -1 <  \frac{x_1}{x_3^{\frac{1}{p}}} 
\leq \frac{x_2}{x_4^{\frac{1}{q}}} < 1\Big\}
\end{equation*}
entirely.
\end{thm}

We see that Theorem~\ref{BellmanTheorem} describes the foliation 
of~$\BellcPlus$: the segments~$\ell_c(a_1,a_2,R)$ are the one dimensional 
subsets of~$\Omegac$ such that~$\BellcPlus$ is linear on each of them. 
Later we will see that this is not all the truth: some of these segments 
form linearity domains of dimension three (see Section~\ref{s34} below). 

The formula
\begin{equation}\label{SymmetryFormula}
\BellcPlus(x_1,x_2,x_3,x_4) = -\BellcMinus(-x_1,x_2,x_3,x_4)
\end{equation} 
leads to the corollary.
\begin{cor}\label{DescriptionOfConvexHull}
The segments~$\ell_c(a_1,a_2,R)$ with~$a_1a_2 < 0$ cover the interior 
of~$\Omegac$ entirely. The function~$\BellcMinus$ is linear on each 
of these segments. Moreover,
\begin{multline*}
\partial \mathbb{K} = \Big(\cup_{R \in [-1,1], a_1,a_2 \in \mathbb{R}} 
\mathfrak{L}(a_1,a_2,R)\Big)\cup
\\ \big\{(x_1,x_2,|x_1|^p,x_4,x_1x_2)\,\big|\;  
x_4\geq |x_2|^q\big\} \cup \big\{(x_1,x_2,x_3,|x_2|^q,x_1x_2)\,\big|\; 
x_3\geq |x_1|^p\big\},
\end{multline*}
where the segments~$\mathfrak{L}(a_1,a_2,R)$ are given by the formula
\begin{multline*}
\mathfrak{L}(a_1,a_2,R)= \Big((\tau -(1-\tau)\rho(R)) a_1, 
(\tau -(1-\tau)\phi(R)) a_2, (\tau +(1-\tau)\rho(R)^p) |a_1|^p,
(\tau +(1-\tau)|R|^p)|a_2|^q, 
\\
 (\tau+(1-\tau)\rho(R)\phi(R))a_1a_2\Big), \quad \tau \in [0,1].
\end{multline*}
\end{cor}

\begin{thm}\label{BellmanTheorem2}
For any~$R \in [-1,1]$,~$a_1,a_2 \in \mathbb{R}$ such that~$a_2 > 0$, 
the function~$\BelldPlus$ is linear on the segment~$\ell_d(a_1,a_2,R)$ 
with the endpoints
\begin{equation*}
a=(a_1,|a_1|^p,|a_2|^q,a_1a_2), 
\qquad b=(-\rho(R)a_1,|\rho(R)a_1|^p,|R|^p|a_2|^q,\rho(R)\varphi(R)a_1a_2).
\end{equation*} 
The function~$\BelldMinus$ is linear on the 
segments~$\ell_d(a_1,a_2,R)$ with~$a_2 < 0$.
\end{thm}
Theorem~\ref{BellmanTheorem} has two assertions. The proof of each of 
them is presented in its own subsection (Subsections~\ref{FirstAssertion} 
and~\ref{SecondAssertion}). Theorem~\ref{BellmanTheorem2} is proved in 
Section~\ref{s8}.

\begin{proof}[Proof of Lemma~\textup{\ref{RealR0}}] 
We rewrite~\eqref{eqR0} as~$\kappa(R_0)=0$, where
\begin{equation*}
\kappa(R) = (p-1)|R|^{\frac{p}{2}-1}(1+R)-1-R|R|^{p-2}, \quad R \in (-1,1).
\end{equation*}
We differentiate~$\kappa$ to find 
\begin{equation*}
\kappa'(R)
=(p-1)|R|^{\frac p2-2}
\Big(\Big(\frac p2-1\Big)\sign R+\frac p2|R|-|R|^{\frac p2}\Big).
\end{equation*}
This is greater than~$0$ when~$R\in(0,1)$.
We also have~$\kappa(0)=-1<0$ and~$\kappa(1)=2(p-2)>0$. Therefore,~$\kappa$ 
has a unique root~$R_0$ in~$(0,1)$. The function~$\kappa$ changes 
sign from negative to positive at~$R_0$.

We also compute
\begin{equation}\label{eqlambdaprime}
\lambda'(R)=\frac{(p-1)^2|R|^{p-2}}{\big(1+R|R|^{p-2}\big)^2}-\frac{1}{(1+R)^2}
=\frac{\kappa(R)\Big((p-1)|R|^{\frac p2-1}(1+R)+
\big(1+R|R|^{p-2}\big)\Big)}{\big(1+R|R|^{p-2}\big)^2(1+R)^2}.
\end{equation}
Thus,~$\sign(\lambda'(R)) = \sign(\kappa(R))$.

Note that~$\kappa'(R) < 0$ when~$R\in (-1,0)$. Moreover,~$\kappa(-1)=0$, 
so~$\kappa(R)<0$ when~$R\in(-1,0)$. We conclude that~$\lambda'(R)<0$ 
 for~$R \in(-1,0)$. 

Thus,~$\lambda$
decreases on~$(-1,R_0)$ and increases on~$(R_0,1)$.
\end{proof}

\subsection{First assertion of Theorem~\ref{BellmanTheorem}}
\label{FirstAssertion}
To prove that the function~$\BellcPlus$ is linear on a certain 
segment~$\ell(a,b)$ that connects two points~$a$ and~$b$ 
on~$\FixedBoundary\Omegac$, we will construct an affine function~$\Psi$ 
(depending on~$a$ and~$b$) such that
\begin{equation*}
\Psi(a) = \BellcPlus(a),\ \Psi(b) = \BellcPlus(b),\hbox{ and } \Psi(x) 
\geq \BellcPlus(x)\quad \hbox{for any } x \in \Omegac.
\end{equation*}
By the third and fourth statements of Lemma~\ref{SimpleProperties}, 
it suffices to verify the inequality~$\Psi(x) \geq \BellcPlus(x)$ 
for~$x\in \FixedBoundary \Omegac$ only. Let
\begin{equation}\label{Psifunction}
\Psi(x) = t_0+t_1 x_1+t_2x_2+t_3x_3+t_4x_4.
\end{equation}
By the preceeding,~$\Psi$ majorizes~$\BellcPlus$ if and only if the 
function~$\Phi\colon \mathbb{R}^2 \to \mathbb{R}$, given by the formula
\begin{equation}\label{eq02}
\Phi(x_1,x_2) = \Psi(x_1,x_2,|x_1|^p,|x_2|^q) -x_1x_2 = 
t_0+t_1 x_1+t_2x_2+t_3|x_1|^p+t_4|x_2|^q-x_1x_2,
\end{equation} 
is non-negative. Moreover, if~$\Phi(a_1,a_2) = \Phi(b_1,b_2) = 0$, 
then~$\BellcPlus$ is linear on the segment~$\ell(a,b)$. 

\begin{lem}\label{lem21}
The function~$\Psi$ defined by~\eqref{Psifunction} 
majorizes~$\BellcPlus$ on~$\Omegac$ if and only if the following two 
conditions hold\textup:
\begin{enumerate}
\item[\rm 1)] $t_3, t_4 >0$\textup;
\item[\rm 2)] $H(x_1) \geq 0$ for any~$x_1 \in \mathbb{R}$, where
\begin{equation}\label{Hfunction}
H(x_1)
=t_0-\frac{1}{p}\frac{|x_1-t_2|^p}{(qt_4)^{p-1}}+t_3|x_1|^p+t_1x_1.
\end{equation} 
\end{enumerate} 
\end{lem}
\begin{proof}
Fix~$x_2 \ne t_1$ and consider the asymptotic behavior of~$\Phi$ as~$x_1$ 
tends to infinity.
The senior term of~$\Phi$ should be non-negative, which leads 
to~$t_3 \geq 0$. The equality~$t_3=0$ 
 is impossible since in this case~$\Phi$ is a linear with respect 
to~$x_1$ function with non-zero senior
 coefficient. Such a function cannot be non-negative on the entire 
axis. Thus,~$t_3>0$. Similarly,~$t_4>0$. 

We fix~$x_1$ and see that the function~$x_2 \mapsto \Phi(x_1,x_2)$ 
attains its minimum at the point 
\begin{equation*}
x_2=\sign(x_1-t_2)\left|\frac{x_1-t_2}{qt_4}\right|^{p-1}.
\end{equation*}
We plug this expression back into~$\Phi$ and find
\begin{equation*}
\min_{x_2}\Phi(x_1,x_2)
=\Phi\left(x_1,\sign(x_1-t_2)\frac{|x_1-t_2|^{p-1}}{(qt_4)^{p-1}}\right)
=H(x_1).
\end{equation*} 
\end{proof}

We study the equations 
\begin{equation}\label{bc}
\Phi(a_1,a_2)= \Phi(b_1,b_2)=0
\end{equation} 
under the condition~$\Phi \geq 0$. In particular,~$(a_1,a_2)$  
and~$(b_1,b_2)$ are minima of~$\Phi$. Thus,
\begin{align}
\frac{\partial \Phi}{\partial x_1}(a_1,a_2)=& \; 
\frac{\partial \Phi}{\partial x_1}(b_1,b_2)=0 \label{bc1};
\\
\frac{\partial \Phi}{\partial x_2}(a_1,a_2)=& \; 
\frac{\partial \Phi}{\partial x_2}(b_1,b_2)=0 \label{bc2}.
\end{align}

The derivative of~$\Phi$ 
satisfies~$\frac{\partial \Phi}{\partial x_1}(x_1,x_2) = 
t_1 +p |x_1|^{p-2}x_1 t_3- x_2$. Note that~\eqref{bc1} does not have 
solutions for which only one of the identities~$a_1=b_1$ or~$a_2=b_2$ 
hold (i.e. either~$a_1=b_1$ and~$a_2 = b_2$ or~$a_1\ne b_1$ and~$a_2 \ne b_2$). 
Consequently,~$a_1\ne b_1$ and~$a_2 \ne b_2$. We solve~\eqref{bc1} 
for~$t_1$ and~$t_3$:
\begin{equation}\label{t1}
t_1 = \frac{a_2b_1|b_1|^{p-2}-a_1b_2|a_1|^{p-2}}{b_1|b_1|^{p-2}-a_1|a_1|^{p-2}};
\end{equation}
\begin{equation}\label{t3}
t_3 = \frac{b_2-a_2}{p(b_1|b_1|^{p-2}-a_1|a_1|^{p-2})}.
\end{equation}
Similarly, we solve~\eqref{bc2} for~$t_2$ and $t_4$:
\begin{equation}\label{t2}
t_2 = \frac{a_1b_2|b_2|^{q-2}-a_2b_1|a_2|^{q-2}}{b_2|b_2|^{q-2}-a_2|a_2|^{q-2}};
\end{equation}
\begin{equation}\label{t4}
t_4 = \frac{b_1-a_1}{q(b_2|b_2|^{q-2}-a_2|a_2|^{q-2})}.
\end{equation}

We also have the system~\eqref{bc} that defines~$t_0$. 
We do not need the expression for~$t_0$, however, the compatibility 
condition is of crucial importance:
\begin{equation}\label{Preeqcup1}
(b_1-a_1)t_1+(b_2-a_2)t_2+(|b_1|^p-|a_1|^p)t_3+
(|b_2|^q-|a_2|^q)t_4+(a_1a_2-b_1b_2)=0.
\end{equation}
Using formulas~\eqref{t1} and~\eqref{t3}, we get
\begin{equation*}
(|b_1|^p-|a_1|^p)pt_3+(b_1-a_1)t_1=b_1b_2-a_1a_2.
\end{equation*}
Similarly, formulas~\eqref{t2} and~\eqref{t4} lead to
\begin{equation*}
(|b_2|^p-|a_2|^p)qt_4+(b_2-a_2)t_2=b_1b_2-a_1a_2.
\end{equation*}
With these identities in hand, we rewrite~\eqref{Preeqcup1} as
\begin{equation}\label{eqcup1}
\frac{1}{q}(b_1-a_1)t_1+\frac{1}{p}(b_2-a_2)t_2=0.
\end{equation}

We treat the~$t_j$ as functions of~$a$ and~$b$ in what follows. We 
summarize our computations in the following lemma.

\begin{lem}\label{lem22}
The function~$\BellcPlus$  is linear on~$\ell(a,b)$ if and only if 
the parameters~$t_1,\,t_2,\,t_3,\,t_4$ given by~\eqref{t1}, \eqref{t2}, 
\eqref{t3}, and~\eqref{t4} satisfy the conditions below.
\begin{enumerate}
\item[\rm1)] $t_3>0$, $t_4 >0$.
\item[\rm2)] $H(x_1)\geq 0$ for any~$x_1\in\mathbb{R}$.
\item[\rm3)] Equality \eqref{eqcup1} holds.
\end{enumerate} 
\end{lem}

To prove the first assertion of Theorem~\ref{BellmanTheorem}, we need to 
restate the conditions of Lemma~\ref{lem22}. As for the first 
condition~$t_3>0$ and~$t_4>0$,
it is equivalent to~$(b_1-a_1)(b_2-a_2)>0$, see 
formulas~\eqref{t3} and~\eqref{t4}. We introduce new variables 
\begin{equation}\label{Rdef}
R_1=-\frac{b_1}{a_1}, \qquad R_2 = -\frac{b_2|b_2|^{q-2}}{a_2|a_2|^{q-2}}.
\end{equation} 
 We express~$t_1,t_2,t_3,t_4$ in the new variables 
(recall~$\varphi(R) = R|R|^{p-2}$):
\begin{align}
t_1 =& \;a_2 \frac{\phi(R_1)-\phi(R_2)}{1+\phi(R_1)}\,;& \qquad
t_2 =& \;a_1 \frac{R_2-R_1}{1+R_2}\,;\label{t12}
\\
pt_3 =& \;\frac{a_2}{a_1|a_1|^{p-2}} \cdot \frac{1+\phi(R_2)}{1+\phi(R_1)}\,;& 
\qquad
qt_4 =& \;\frac{a_1}{a_2|a_2|^{q-2}} \cdot \frac{1+R_1}{1+R_2}\,.\label{t34}
\end{align}
Note that division by~$a_1$ and~$a_2$ is eligible since~$a_1a_2 > 0$ in 
Theorem~\ref{BellmanTheorem}.

\begin{lem}\label{Eqcup3Lem}
The third condition in Lemma~\textup{\ref{lem22}} is equivalent to 
\begin{equation}\label{eqcup3}
\lambda(R_1)=\lambda(R_2).
\end{equation}
\end{lem}
\begin{proof}
We divide~\eqref{eqcup1} by~$a_1a_2$, express everything in terms 
of~$R_1$ and~$R_2$, and obtain
\begin{equation*}
\frac{1}{q}(1+R_1)
\frac{\phi(R_1)-\phi(R_2)}{1+\phi(R_1)}+
\frac{1}{p}(1+\phi(R_2))\frac{R_2-R_1}{1+R_2}=0,
\end{equation*}
which, after division by~$(1+R_1)(1+\phi(R_2))/p$ appears to be 
$\lambda(R_1)=\lambda(R_2)$.
\end{proof}

\begin{lem}\label{Hsign}
If~$a$ and~$b$ are such that~$t_3,\,t_4>0$, then~$H$ does not change 
sign on the real line.
\end{lem}
\begin{proof}
We have chosen the~$t_i$ in such a way that~$H(a_1)=H(b_1)$. 
Moreover,~$H'(a_1)=H'(b_1)=0$ 
since~$\nabla \Phi (a_1,a_2) = \nabla \Phi(b_1,b_2) = 0$. Therefore,~$H''$ 
has at least two distinct roots on~$(a_1,b_1)$. Note that the function
\begin{equation*}
H''(x_1)=(p-1)\left(pt_3|x_1|^{p-2}-\frac{|x_1-t_2|^{p-2}}{(qt_4)^{p-1}}\right)
\end{equation*}
is either equal to zero, or has no more than two roots. If~$H''$ is 
identically zero, then~$H$ vanishes on~$\mathbb{R}$ as well. In the 
other case,~$H$ has unique local extremum on~$(a_1,b_1)$. In this 
case,~$H''$ has exactly two roots. Therefore, ~$H''(a_1)$ and~$H''(b_1)$ 
have the same sign and~$H$ does not change sign on~$\mathbb{R}$.
\end{proof}

\begin{lem}\label{PsiLemma}
Let~$a$ and~$b$ be such that~$t_3,\,t_4>0$. If~$R_1=R_2$, then~$H$ is 
identically zero. If~$R_1\ne R_2$, then $H\geq 0$ is equivalent to
$\psi(R_2)>\psi(R_1)$, where~$\psi(R)=\frac{|1+\phi(R)|}{|1+R|^{p-1}}$.
\end{lem}
\begin{proof}
Identity~\eqref{t34} leads to
\begin{equation*}
pt_3(qt_4)^{p-1}=\frac{\psi(R_2)}{\psi(R_1)}.
\end{equation*}

If~$R_1=R_2$, then~$t_1=t_2=0$ and~$pt_3(qt_4)^{p-1}=1$. Thus,~$H$ is 
identically equal to zero in this case. 

If~$R_1\ne R_2$, then~$t_2 \ne 0$. In the case~$pt_3(qt_4)^{p-1}=1$, the 
signs of~$H$ at~$-\infty$ and~$+\infty$ differ, which contradicts 
Lemma~\ref{Hsign}. If~$pt_3(qt_4)^{p-1}>1$, 
then~$H$ is positive at the infinities, if~$pt_3(qt_4)^{p-1}<1$ it is negative.
\end{proof}
The first assertion of Theorem~\ref{BellmanTheorem} is almost proved. 
Indeed, consider any segment~$\ell_c(a_1,a_2,R)$ with~$a_1a_2 > 0$. 
Define~$R_1$ and~$R_2$ by formulas~\eqref{Rdef}, i.e.~$R_1 = \rho(R)$ 
and~$R_2 = R$, and note that~$R_1 \geq R_2$. 
Then,~$\lambda(R_1) = \lambda(R_2)$ (by Definition~\ref{Rho}) 
and~$\psi(R_2)>\psi(R_1)$ since~$\psi$ decreases on~$(-1,1)$. Thus, 
by Lemmas~\ref{Eqcup3Lem} and~\ref{PsiLemma}, the function~$\Psi$ 
defined by~\eqref{Psifunction},~\eqref{t12}, and~\eqref{t34} satisfies 
the requirements of Lemma~\ref{lem22}. This means~$\BellcPlus$ is 
linear on~$\ell_c(a_1,a_2,R)$. It remains to combine 
formulas~\eqref{corabpoints} and~\eqref{Linearity} to see 
that~$\BellcPlus(x) = x_3^{\frac{1}{p}}x_4^{\frac{1}{q}}$ 
if~$x \in \ell_c(a_1,a_2,R)$ and~$R \geq R_0$.

\subsection{Second assertion of Theorem~\ref{BellmanTheorem}}
\label{SecondAssertion}
\begin{defn}\label{TDef}
Consider the mapping
\begin{equation*}
T\colon\Omegac\setminus \Big(\{x\mid x_1=x_3=0\}
\cup\{x \mid x_2=x_4=0\}\Big)\to [-1,1]^2
\end{equation*}
given by the rule 
\begin{equation}\label{MappingT}
T\colon (x_1,x_2,x_3,x_4)\mapsto (x_1x_3^{-\frac1p},x_2x_4^{-\frac1q}).
\end{equation}
\end{defn}
The skeleton~$\FixedBoundary\Omegac$ is mapped onto the vertices~$(\pm1,\pm1)$.
Due to the homogeneity of~$\BellcPlus$ (Lemma~\ref{SimpleProperties}), 
it suffices to compute the 
values of~$\BellcPlus$ for~$x$ such that~$x_3=x_4=1$; the values at all 
other points may be restored from them:
\begin{equation*}
\BellcPlus(x)=x_3^{\frac1p}x_4^{\frac1q}
\BellcPlus(x_1x_3^{-\frac1p},x_2x_4^{-\frac1q},1,1)\,. 
\end{equation*}
Therefore, to prove the second assertion of Theorem~\ref{BellmanTheorem}, 
it suffices to show that the union of~$T(\ell_c(a_1,a_2,R))$ over 
all~$a_1$,~$a_2$, and~$R$ such that~$a_1>0$,~$a_2 > 0$, and~$R\in (-1,1)$, 
covers the triangle
\begin{equation}\label{Triangle}
\{(y_1,y_2)\mid -1<y_2<y_1<1\}.
\end{equation}
Note that~$T(\ell_c(a_1,a_2,1)) = \{(y_1,y_1)\mid y_1 \in [-1,1]\}$.

\begin{lem}\label{lemeta}
The function~$\eta$ defined by the formula  
\begin{equation*}
\eta\colon (\tau,R) \mapsto \Big(\eta\cii{1}(\tau,R),\eta\cii{2}(\tau,R)\Big)
\df\left(\frac{-\tau+(1-\tau)\rho(R)}{(\tau+(1-\tau)\rho(R)^p)^{\frac{1}{p}}},
\frac{-\tau+(1-\tau)\phi(R)}{(\tau+(1-\tau)|R|^p)^{\frac{1}{q}}}\right)
\end{equation*}
maps~$(0,1)\times(-1,1)$ onto the triangle~\eqref{Triangle} bijectively.
\end{lem}

\begin{rem}
For~$R$ fixed, the image of the mapping~$\tau \mapsto \eta(\tau,R)$ 
coincides with~$T(\ell_c(a_1,a_2,R))$, here~$a_1 > 0$ and~$a_2 > 0$.
\end{rem}
\begin{proof}
First, we show that for any~$(y_1,y_2) \in(-1,1)^2$ such that~$y_2 < y_1$, 
there exist~$R\in(-1,1)$ and~$\tau \in (0,1)$ such 
that~$y_i=\eta\cii{i}(\tau,R)$,~$i=1,2$.

Note that~$\frac{\partial \eta\cii{1}}{\partial \tau} (\tau,R)<0$:
\begin{equation}\label{eta1tau}
\begin{split}
\pd{\eta_1}{\tau}&=\Big[\frac{1}{p}\big(\tau-(1-\tau)\rho\big)(1-\rho^p)-(1+\rho)\big(\tau+(1-\tau)\rho^p\big)\Big]
\big(\tau+(1-\tau)\rho^p\big)^{-1-\frac{1}{p}}=\\
&=-\Big[\Big(\frac{1}{q}+\rho+\frac{1}{p}\rho^p\Big)\tau
+\Big(\frac{1}{p}(1-\rho^p)\rho+(1+\rho)\rho^p\Big)(1-\tau)\Big]
\big(\tau+(1-\tau)\rho^p\big)^{-1-\frac{1}{p}}<0.
\end{split}
\end{equation}

Moreover,~$\eta\cii1(0,R)=1$,~$\eta\cii1(1,R)=-1$. Therefore, 
for any~$y_1 \in(-1,1)$ 
and~$R \in [-1,1]$, there exists a unique value~$\tau \in (0,1)$ such 
that~$\eta\cii1(\tau,R)=y_1$. Let us denote this value by the 
symbol~$\tau_1(R,y_1)$. The function~$\tau_1(\cdot,y_1)$ is continuous 
for any fixed~$y_1\in(-1,1)$. What is more,~$\eta(\tau,1) = 
(1-2\tau,1-2\tau)$ and~$\eta(\tau,-1)=(1-2\tau,-1)$. 
Thus,~$\tau_1(-1,y_1)=\tau_1(1,y_1)=\frac{1-y_1}{2}$.

Fix~$y_1$ for a while. Note that the function
\begin{equation*}
Y_2\colon R \mapsto \eta\cii2(\tau_1(R,y_1),R)
\end{equation*}
is continuous and 
\begin{equation*}
Y_2(-1)=\eta\cii2\Big(\frac{1-y_1}{2},-1\Big)=-1< y_2 < y_1
=\eta\cii2\Big(\frac{1-y_1}{2},1\Big) = Y_2(1).
\end{equation*}
Thus, for some~$R\in(-1,1)$ we have~$Y_2(R)=y_2$. This means that the 
identity~$\eta(\tau,R)=(y_1,y_2)$ holds for~$\tau = \tau_1(R,y_1)$ and 
this specific choice of~$R$. Thus, we have proved that the union 
of~$T(\ell_c(a_1,a_2,R))$ covers the triangle~\eqref{Triangle}.

Let us now show that the function~$Y_2$ increases provided~$y_1$ is 
fixed. If this is true, then for any~$y_1\in(-1,1)$ and~$y_2\in(-1,y_1)$, 
there exists unique value~$R$ such that~$Y_2(R)=y_2$. Moreover, the 
same monotonicity leads to the inequality~$\eta_2 \leq \eta_1$. 
Thus,~$\eta$ is bijective.

So it remains to prove the mentioned monotonicity. We compute the derivative:
\begin{equation*}
Y_2'(R)= \Big(\pd{\eta_1}{\tau}\Big)^{-1}\!\!\!\cdot \Delta(R), 
\qquad \Delta(R)\df \pd{\eta_2}{R}\pd{\eta_1}{\tau}-
\pd{\eta_2}{\tau}\pd{\eta_1}{R}.
\end{equation*}
We have already proved~$\pd{\eta_1}{\tau}<0$ (see~\eqref{eta1tau}).
 
We investigate~$\Delta$. We compute the remaining partial derivatives:
\begin{equation}\label{eta1R}
\pd{\eta_1}{R}=\tau(1-\tau)\rho'\big(1+\rho^{p-1}\big)
\big(\tau+(1-\tau)\rho^p\big)^{-1-\frac{1}{p}},
\end{equation}
\begin{equation}\label{eta2tau}
\pd{\eta_2}{\tau}=\Big[\frac{1}{q}
\big(\tau-(1-\tau)\phi\big)(1-|R|^p)-(1+\phi)\big(\tau+(1-\tau)|R|^p\big)\Big]
\big(\tau+(1-\tau)|R|^p\big)^{-1-\frac{1}{q}},
\end{equation}
\begin{equation}\label{eta2R}
\pd{\eta_2}{R}=(p-1)\tau(1-\tau)(1+R)|R|^{p-2}
\big(\tau+(1-\tau)|R|^p\big)^{-1-\frac{1}{q}}.
\end{equation}

We plug these expressions into the formula for~$\Delta$:
\begin{equation}\label{Delta1}
\begin{split}
\Delta \cdot&\big(\tau+(1-\tau)|R|^p\big)^{1+\frac{1}{q}}
\big(\tau+(1-\tau)\rho^p\big)^{1+\frac{1}{p}}=\\
&=\tau(1-\tau)\Big\{
(p-1)(1+R)|R|^{p-2}\Big[\frac{1}{p}
\big(\tau-(1-\tau)\rho\big)(1-\rho^p)-
(1+\rho)\big(\tau+(1-\tau)\rho^p\big)\Big]-
\\
&\hskip 63pt-\rho'(1+\rho^{p-1})\Big[\frac{1}{q}
\big(\tau-(1-\tau)\phi\big)(1-|R|^p)-(1+\phi)\big(\tau+(1-\tau)|R|^p\big)\Big]	
\Big\}.
\end{split}
\end{equation}
When~$R \in[R_0,1]$, we have~$\rho=R$. Thus, the expression in the 
formula~\eqref{Delta1} is
\begin{equation}\label{Delta2}
\begin{split}
\tau(1-\tau)\Big\{
(p-1)(1+R)R^{p-2}\Big[\frac{1}{p}(\tau-(1-\tau)R)
\big(1-R^p\big)-(1+R)\big(\tau+(1-\tau)R^p\big)\Big]&-\\
	-(1+R^{p-1})\Big[\frac{1}{q}
\big(\tau-(1-\tau)R^{p-1}\big)\big(1-R^p\big)
-\big(1+R^{p-1}\big)\big(\tau+(1-\tau)R^p\big)\Big]&\Big\}=\\
=\frac{1}{p}\tau(1-\tau)\big(\tau+(1-\tau)R^p\big)
\Big[\big(1+R^{p-1}\big)^2-(p-1)^2R^{p-2}(1+R)^2\Big]&<0
\end{split}
\end{equation}
by Lemma~\ref{RealR0}, since the latter expressions in the brackets 
has the sign opposite to~$\lambda'(R)$, see~\eqref{eqlambdaprime}.
Thus, we have proved that~$Y_2$ increases on~$(R_0,1)$ for any fixed~$y_1$.

It remains to consider the case~$R \in (-1,R_0)$. The expression in 
the braces in~\eqref{Delta1} is a linear function of~$\tau$. Thus, 
it suffices to investigate its signs at the endpoints~$\tau=0$ and~$\tau=1$. 

When~$\tau=1$, the expression in the first brackets in~\eqref{Delta1} 
is negative: 
\begin{equation*}
\frac{1-\rho^p}{p}-(1+\rho)=-\frac{1}{q}-\frac{\rho^p}{p}-\rho<0.
\end{equation*} 
The expression in the second brackets in~\eqref{Delta1} is equal to
\begin{equation*}
\frac{1-|R|^p}{q}-(1+\phi)=
-\Big(\frac{1}{p}+\frac{|R|^p}{q}+R|R|^{p-2}\Big)\leq 0.
\end{equation*}
Finally,~$\rho'<0$, therefore, the expression in the braces 
in~\eqref{Delta1} is negative. 

It remains to study the case~$\tau=0$. In this case, our 
expression is equal to
\begin{equation}\label{eq13}
\Upsilon=-(p-1)(1+R)|R|^{p-2}\Big[\frac{\rho}{p}
\big(1-\rho^p\big)+(1+\rho)\rho^p\Big]
+\rho'(1+\rho^{p-1})\Big[\frac{\phi}{q}\big(1-|R|^p\big)+(1+\phi)|R|^p\Big].
\end{equation}

Now we will compute the expressions in the brackets in~\eqref{eq13} 
and the derivative~$\rho'$ separately.
Let us first rewrite the identity~\eqref{lambdadef} in a more convenient form:
\begin{equation}\label{eq14}
\frac{p-1}{1+\phi}=\frac{1-(1+R)\lambda}{1+R}, \qquad 
\phi=\frac{(1+R)(p-1+\lambda)-1}{1-(1+R)\lambda}.
\end{equation}
Similarly, we rewrite the identity~$\lambda(\rho)=\lambda$ as
\begin{equation}\label{eq15}
\frac{p-1}{1+\rho^{p-1}}=\frac{1-(1+\rho)\lambda}{1+\rho}, \qquad 
\rho^{p-1}=\frac{(1+\rho)(p-1+\lambda)-1}{1-(1+\rho)\lambda}.
\end{equation}

We also re-express~$\lambda'(R)$:
\begin{multline}\label{eqlambda'}
\lambda'(R)=\frac{(p-1)^2}{(1+\phi)^2}|R|^{p-2}-\frac{1}{(1+R)^2}=
\Big(\frac{1-(1+R)\lambda}{1+R}\Big)^2\frac{\phi}{R}-\frac{1}{(1+R)^2}=\\
=\frac{\big(1-(1+R)\lambda\big)
\big((1+R)(p-1+\lambda)-1\big)}{R(1+R)^2}-\frac{1}{(1+R)^2}=\\
=\frac{1}{R(1+R)^2}\Big((1+R)(p-1+\lambda)
\big(1-(1+R)\lambda\big)-1+(1+R)\lambda-R\Big)=\\
=\frac{(p-1+\lambda)\big(1-(1+R)\lambda\big)-1+\lambda}{R(1+R)}.
\end{multline}

The identity~$\lambda(R)=\lambda(\rho(R))$ leads to
\begin{equation}\label{rho'}
\rho' = \frac{\lambda'(R)}{\lambda'(\rho)}
=\frac{\rho(1+\rho)}{R(1+R)}\cdot \frac{(p-1+\lambda)
\big(1-(1+R)\lambda\big)-1+\lambda}{(p-1+\lambda)
\big(1-(1+\rho)\lambda\big)-1+\lambda}.
\end{equation}

We rewrite the expression in the first brackets of~\eqref{eq13}:
\begin{multline}\label{eq13sq1}
\frac{\rho}{p}\big(1-\rho^p\big)+(1+\rho)\rho^p = 
\frac{\rho+\rho^p}{p}+\frac{\rho^{p}+\rho^{p+1}}{q}
=\frac{\rho}{q}\Big(\frac{1+(1+\rho)(p-1)}{p-1}(1+\rho^{p-1})-(1+\rho)\Big)=
\\
=\frac{\rho}{q}\Big(
\frac{1+(1+\rho)(p-1)}{1-(1+\rho)\lambda}(1+\rho)-(1+\rho)\Big)=
\frac{\rho(1+\rho)}{q\big(1-(1+\rho)\lambda\big)}
\Big(1+(1+\rho)(p-1)-1+(1+\rho)\lambda\Big)=\\
=\frac{\rho(1+\rho)^2(p-1+\lambda)}{q\big(1-(1+\rho)\lambda\big)}.
\end{multline}

And we also rewrite the expression in the second brackets of~\eqref{eq13}:
\begin{multline}\label{eq13sq2}
\frac{\phi}{q}(1-|R|^p)+(1+\phi)|R|^p=
\frac{\phi+|R|^p}{q}+\frac{|R|^p(1+\phi)}{p}=
\frac{\phi(1+R)}{q}+\frac{\phi R(1+\phi)}{p}=\\
=\frac{\phi(1+R)}{q}+\frac{\phi R(p-1)(1+R)}{p\big(1-(1+R)\lambda\big)}
=\frac{\phi(1+R)}{q}\Big(1+\frac{R}{1-(1+R)\lambda}\Big)=
\frac{\phi(1+R)^2(1-\lambda)}{q\big(1-(1+R)\lambda\big)}.
\end{multline}

Combining~\eqref{rho'},~\eqref{eq13sq1}, and~\eqref{eq13sq2}, 
we re-express~\eqref{eq13} as
\begin{multline}\label{eq16}
\Upsilon=-\frac{(p-1)(1+R)|R|^{p-2}
\rho(1+\rho)^2(p-1+\lambda)}{q\big(1-(1+\rho)\lambda\big)}+
\\
+\frac{\rho(1+\rho)}{R(1+R)}\cdot \frac{(p-1+\lambda)
\big(1-(1+R)\lambda\big)-1+\lambda}{(p-1+\lambda)
\big(1-(1+\rho)\lambda\big)-1+\lambda}\cdot
\frac{(p-1)(1+\rho)}{\big(1-(1+\rho)\lambda\big)}\cdot
\frac{\phi(1+R)^2(1-\lambda)}{q\big(1-(1+R)\lambda\big)}=
\\
=\frac{(p-1)(1+R)|R|^{p-2}\rho(1+\rho)^2}{q\big(1-(1+\rho)\lambda\big)}\times
\\ \left\{-(p-1+\lambda)+
\frac{(p-1+\lambda)\big(1-(1+R)\lambda\big)-1+\lambda}{(p-1+\lambda)
\big(1-(1+\rho)\lambda\big)-1+\lambda}\cdot
\frac{1-\lambda}{1-(1+R)\lambda}\right\}.
\end{multline}
Note that~$\lambda$ attains its maximal value at the endpoints of~$[-1,1]$. 
It equals~$-\frac{p-2}{2}<0$ there, thus,~$\lambda<0$. Consequently, 
the quantities~$1-(1+R)\lambda$ and~$1-(1+\rho)\lambda$ are non-negative. 
Thus, the first multiple in~\eqref{eq16} is non-negative. 

Consider the expression inside the braces in~\eqref{eq16} now. Note 
that both denominators are positive. We have just discussed the second, 
as for the first, its sign coincides with the sign of~$\lambda'(\rho)$ 
(see~\eqref{eqlambda'}), which is positive by Lemma~\ref{RealR0}. 
Multiplying the expression in the braces by the denominators, we get
\begin{multline*}
\sign(\Upsilon)=\sign
\left\{\Big((p-1+\lambda)\big(1-(1+R)\lambda\big)-1+\lambda\Big)
(1-\lambda)\right.-
\\
\left. -(p-1+\lambda)\Big((p-1+\lambda)
\big(1-(1+\rho)\lambda\big)-1+\lambda\Big)\big(1-(1+R)\lambda\big) \right\}=
\\
-\sign\Big\{(1-\lambda)^2-2(1-\lambda)(p-1+\lambda)
\big(1-(1+R)\lambda\big)+
\\ 
(p-1+\lambda)^2\big(1-(1+R)\lambda\big)\big(1-(1+\rho)\lambda\big)\Big\}=
\\
-\sign\bigg\{\Big((1-\lambda)-(p-1+\lambda)\big(1-(1+R)\lambda\big)\Big)^2+
\\
(p-1+\lambda)^2\big(1-(1+R)\lambda\big)
\Big[1-(1+\rho)\lambda-1+(1+R)\lambda\Big]\bigg\}=
\\
-\sign\left\{\Big((1-\lambda)-(p-1+\lambda)
\big(1-(1+R)\lambda\big)\Big)^2+(p-1+\lambda)^2
\big(1-(1+R)\lambda\big)(R-\rho)\lambda\right\}=-1,
\end{multline*}
since~$\lambda<0$ and~$R<\rho$.

We proved that~$\Upsilon<0$ for~$R\in(-1,R_0)$. Therefore,~$\Delta<0$, 
and~$Y_2'>0$. 
\end{proof}

\subsection{The structure of the foliation}\label{s34}
\begin{prop}\label{propfol}
The interior of~$\Omegac$ is foliated by one and three dimensional 
extremals of~$\BellcPlus$. 
The three dimensional domains are parametrized by~$t_3>0$, each such 
domain is the convex hull of the curve~$\gamma_{t_3}$\textup: 
\begin{equation*}
\gamma_{t_3} \colon x_1\mapsto (x_1,pt_3 \phi(x_1), |x_1|^p, 
|x_1|^{p}(pt_3)^q), \quad x_1 \in \mathbb{R}.
\end{equation*}
The~$T$-image of each such linearity domain is  the subdomain 
of~$[-1,1]^2$ bounded by the curve~$\eta_-=\eta(\cdot,R_0)$ and 
its symmetric with respect to~$(0,0)$ image~$\eta_+$ 
\textup(see Figure~\textup{\ref{figtrivdom})}.

The remaining part of the domain is covered by one dimensional extremals. 
The~$T$-image of each such segment is either the 
curve~$\eta(\cdot,R)$,~$R\in(-1,R_0)$, or its symmetric with 
respect to~$(0,0)$ image. 
\end{prop}
\begin{figure}[h!]
\begin{center}
\includegraphics[scale=0.4]{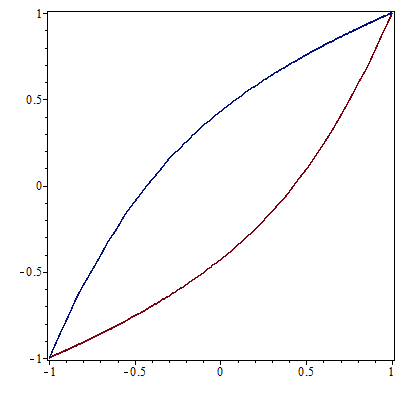}
\caption{The curves~$\eta_-$ and~$\eta_+$ bound the~$T$-image of 
the domain where~$\BellcPlus(x) = x_3^{\frac{1}{p}}x_4^{\frac{1}{q}}$.}
\label{figtrivdom}
\end{center}
\end{figure}

\begin{lem}\label{lemmatriangle}
Let~$t_3>0$. The function~$\BellcPlus$ is linear on the convex 
hull of~$\gamma_{t_3}$, moreover,~$\BellcPlus = 
x_3^{\frac{1}{p}}x_4^{\frac{1}{q}}$ there. 

Let~$a,b,c \in \FixedBoundary\Omegac$ be three distinct points such that 
their convex hull does not lie inside the topological boundary 
of~$\Omegac$. If~$\BellcPlus$ coincides with a linear 
function~$t_0+t_1x_1+t_2x_2+t_3x_3+t_4x_4$ on the said convex hull, 
then~$t_0=t_1=t_2=0$,~$pt_3(qt_4)^{p-1}=1$, and~$a,b,c \in\gamma_{t_3}$.
\end{lem}
\begin{proof}
Note that
\begin{equation*}
\BellcPlus(x)=x_1x_2=x_3^{\frac{1}{p}}x_4^{\frac{1}{q}} = t_3 x_3+t_4x_4, 
\quad \hbox{where~$t_4$ is such that}\quad pt_3(qt_4)^{p-1}=1,
\end{equation*} 
for any~$x\in\gamma_{t_3}$. What is more, for 
any~$x \in \FixedBoundary \Omegac$, there is an inequality
\begin{equation*}
\BellcPlus(x)=x_1x_2 \leq \big|pt_3 |x_1|^p\big|^{\frac1p}
\big|qt_4|x_2|^q\big|^{\frac1q}\leq t_3 |x_1|^p+t_4|x_2|^q=t_3x_3+t_4x_4.
\end{equation*}
Thus,~$\BellcPlus(x) \leq t_3x_3+t_4x_4$ for $x\in \Omegac$. On the 
other hand, it follows from concavity of~$\BellcPlus$ 
that~$\BellcPlus(x) \geq t_3x_3+t_4x_4$ on~$\conv(\gamma_{t_3})$. 
Therefore,~$\BellcPlus(x) = t_3x_3+t_4x_4 = x_3^{\frac{1}{p}}x_4^{\frac{1}{q}}$ 
on~$\conv(\gamma_{t_3})$ since~$x_4 = (pt_3)^qx_3$ there. The first 
assertion of the lemma is proved.

By Lemma~\ref{lem22}, the parameters~$t_1,\,t_2,\,t_3,\,t_4$ should fall 
under~\eqref{t1},~\eqref{t3},~\eqref{t2}, and~\eqref{t4} for each of 
the pairs~$(a,b)$,~$(a,c)$, and~$(b,c)$. Consequently, the 
numbers~$a_1,b_1$, and~$c_1$ are distinct since~$a,b,c$ are 
distinct (if, say,~$a_1=b_1$, then the 
equations~\eqref{t1},~\eqref{t3},~\eqref{t2}, and~\eqref{t4} for the 
pairs~$(a,c)$ and~$(b,c)$ lead to the conclusion~$a_2=b_2$). What is 
more,~$t_3>0$ and~$t_4>0$. The function~$H$ given in~\eqref{Hfunction}, 
satisfies~$H(a_1)=H(b_1)=H(c_1)=H'(a_1)=H'(b_1)=H'(c_1)=0$, and thus 
equals to zero (see the proof of Lemma~\ref{Hsign}). 
Therefore,~$t_0=t_1=t_2=0$ and~$pt_3(qt_4)^{p-1}=1$. 
Further,~$t_3|a_1|^p+t_4|a_2|^q=B(a)=a_1a_2$. This is the case of equality 
in Young's inequality, which leads to~$a_2=pt_3\phi(a_1)$. This 
means~$a \in\gamma_{t_3}$. Similarly,~$b$ and~$c$ also lie on the same curve.
\end{proof}

\begin{rem}\label{remconv}
The convex hulls of the curves~$\gamma_{t_3}$,~$t_3>0$, are pairwise disjoint.
\end{rem}

\begin{lem}\label{lemconvgamma}
The image of the convex hull of~$\gamma_{t_3}$ under~$T$ is the region 
bounded by~$\eta_-$ and~$\eta_+$.
\end{lem}
\begin{proof}
Let the~$T$-image of~$x \in \conv\gamma_{t_3}$ lie below the main 
diagonal of~$[-1,1]^2$. We know that for some~$R \in [-1,1]$, the 
curve~$\eta(\cdot, R)$ passes through~$T(x)$. Therefore, there exists a 
chord with the endpoints~$a$ and~$b$, defined by~\eqref{corabpoints}, 
which contains~$x$. Lemma~\ref{lemmatriangle} and Remark~\ref{remconv} 
lead to the inclusion~$a,b\in \gamma_{t_3}$. Thus,
\begin{equation*}
\frac{a_2}{\phi(a_1)}=pt_3=\frac{b_2}{\phi(b_1)}
=\frac{\phi(R)a_2}{\phi(\rho(R))\phi(a_1)},
\end{equation*}
which leads to~$\rho(R)=R$, which is~$R \in [R_0,1]$.

If~$R\in [R_0,1]$, then with the choice~$a_2=\phi(a_1)pt_3$, the 
points~$a$ and~$b$ given by~\eqref{corabpoints}, lie on~$\gamma_{t_3}$. 
Consequently, the chord that connects them lies inside~$\conv(\gamma_{t_3})$. 
Therefore, the part of the set~$T(\conv(\gamma_{t_3}))$ that lies below 
the diagonal of~$[-1,1]^2$, coincides with the 
set~$\{\eta(\tau,R)\colon \tau \in[0,1], R \in [R_0,1]\}$. 
It remains to notice that the latter set is exactly the region 
bounded by~$\eta_-$ and the main diagonal (this follows from the 
bijectivity of~$T$ and the monotonicity of~$Y_2$, see the proof 
of Lemma~\ref{lemeta}).
\end{proof}

\begin{proof}[Proof of Proposition~\textup{\ref{propfol}}]
Proposition~\ref{propfol} follows from Lemmas~\ref{lemmatriangle} 
and~\ref{lemconvgamma}.
\end{proof}

\section{The computation of~$c_{p,r}^{\star,\mathbb{R}}$}
\label{s4}
By Proposition~\ref{CFromBell} and Lemma~\ref{SimpleProperties},
\begin{equation*}
(c_{p,r}^{\star,\mathbb{R}})^{-1} = 
\sup_{|x_1|^p< x_3}\frac{|\BellcPlus(x_1,0,x_3,x_4)|^r}{(x_3^{r/p}-|x_1|^r)x_4^{r/q}}.
\end{equation*}

Consider the segment~$\ell_c(a_1,a_2, R)$ with the 
endpoints~$a=(-1,-1,1,1)$ and~$b=(\rho(R),\phi(R),\rho(R)^p,|R|^p)$, 
defined by the parameter~$R\in(-1,1)$. We may restrict our attention to 
such segments only, due to homogeneity considerations. Note that 
if~$R<0$, then~$\ell_c(a_1,a_2,R)$ does not contain points~$x$ such 
that~$x_2=0$. If~$R\geq 0$ such a point~$\tau a +(1-\tau)b$ corresponds 
to the value~$\tau = \frac{\phi(R)}{1+\phi(R)}$ and its coordinates are 
\begin{equation}\label{eqcoord}
x_1 = \frac{\rho(R)-\phi(R)}{1+\phi(R)},
\quad x_2=0,
\quad x_3=\frac{\rho(R)^p+\phi(R)}{1+\phi(R)},
\quad x_4=\frac{R^p+\phi(R)}{1+\phi(R)}.
\end{equation}
The value of~$\BellcPlus$ at this point is
\begin{equation}\label{eqval}
\BellcPlus(\tau a + (1-\tau) b) = \frac{(\rho(R)+1)\phi(R)}{1+\phi(R)}
\end{equation}
according to Theorem~\ref{BellmanTheorem}.

Thus, we need to maximize 
\begin{multline*}
M_r(R)=\frac{|\BellcPlus(x_1,0,x_3,x_4)|^r}{x_3^{r/p}-|x_1|^r} \cdot x_4^{-r/q} 
=\\ \left(\frac{(\rho(R)+1)\phi(R)}{1+\phi(R)}\right)^r \cdot 
\frac{(1+\phi(R))^{r/q}}{(R^p+\phi(R))^{r/q}
\Big(\big(\frac{\rho(R)^p+\phi(R)}{1+\phi(R)}\big)^{r/p}-
\big|\frac{\phi(R)-\rho(R)}{1+\phi(R)}\big|^r\Big)}\,,
\end{multline*}
when~$R \in [0,1]$. We set~$r=p$ and concentrate on the proof of 
Theorem~\ref{Cthm} in the real-valued case. Let also~$M = M_p$. 
Note that~$\rho(R)\geq R\geq R^{p-1}=\phi(R)\geq 0$. 
We slightly modify the expression~$M(R)$. Consider the function
\begin{equation*}
Q\colon (t,s) \mapsto \frac{(s^p+t)(1+t)^{p-1}-(s-t)^p}{t(1+s)^p}, 
\quad s\geq t>0.
\end{equation*}
We extend it to the case~$t=0$ by continuity.

By an elementary computation, 
\begin{equation*}
\frac{1}{M(R)}= Q\big(R^{p-1},\rho(R)\big)\cdot 
\left(\frac{1+R}{1+R^{p-1}}\right)^{p-1} 
\geq Q\big(R^{p-1},\rho(R)\big).
\end{equation*}

Let us show that the latter expression attains its minimum at~$R=0$. Clearly, 
\begin{equation*}
\frac{1}{M(0)}= Q\big(0,\rho(0)\big).
\end{equation*}

We investigate~$Q$. Fix~$s>0$ and consider the 
function~$h_1(t)\df(1+s)^p Q(t,s)$. By 
continuity,~$h_1(0) = (p-1)s^p+1+ps^{p-1}$. 
We want to show~$h_1(t)\geq h_1(0)$ when~$t \geq 0$. This is 
equivalent to~$h_2(t)\df th_1(t)-th_1(0)\geq 0$. We differentiate~$h_2$ 
and get~$h_2'(0)=0$ and
\begin{multline}\label{eqh2}
h_2''(t) = \frac{\partial^2 }{\partial t^2}\Big(t(1+s)^pQ(t,s)\Big) 
= \frac{\partial^2 }{\partial t^2}\Big((s^p+t)(1+t)^{p-1}-(s-t)^p\Big)=
\\
2(p-1)(1+t)^{p-2}+(p-1)(p-2)(s^p+t)(1+t)^{p-3}-p(p-1)(s-t)^{p-2}. 
\end{multline}
By Young's inequality,
\begin{equation*}
\frac{2}{p}(1+t)^{p-2}+\frac{p-2}{p}s^p(1+t)^{p-3}
\geq s^{\frac{p-2}{p}p}(1+t)^{\frac{2}{p}(p-2)+\frac{p-2}{p}(p-3)}
=s^{p-2}(1+t)^{\frac{(p-1)(p-2)}{p}}\geq s^{p-2}\geq (s-t)^{p-2}.
\end{equation*}
Taking into account~\eqref{eqh2}, this leads to~$h_2''(t) \geq 0$. 
Thus, the function~$h_2$ is convex when~$t\geq 0$ and satisfies the 
identities~$h_2(0)=h_2'(0)=0$. 
Consequently,~$h_2(t)\geq 0$,~$h_1(t)\geq h_1(0)$, and 
\begin{equation*}
Q(t,s)\geq Q(0,s)=  \frac{(p-1)s^p+1+ps^{p-1}}{(1+s)^p}.
\end{equation*}
Let us find the minimum (with respect to~$s$) of $Q(0,s)$. We differentiate:
\begin{equation*}
\frac{\partial}{\partial s}Q(0,s) 
=\frac{p(p-1)s^{p-2}}{(1+s)^{p-1}}-p\frac{(p-1)s^p+1+ps^{p-1}}{(1+s)^{p+1}}
=p\frac{(p-1)s^{p-2}+(p-2)s^{p-1}-1}{(1+s)^{p+1}}.
\end{equation*}
There exists a unique positive point where the latter expression changes 
sign from negative to positive, call it~$s_0$. Then~$s_0$ 
satisfies~\eqref{eqs0}, which might be further rewritten in terms of~$\lambda$ 
\begin{align*}
\label{eqs0}
(p-1)s_0^{p-2}(1+s_0)&=1+s_0^{p-1};
\\
(p-1)-\frac{p-1}{1+s_0^{p-1}}=\frac{(p-1)s_0^{p-1}}{1+s_0^{p-1}} 
&= \frac{s_0}{1+s_0}=1-\frac{1}{1+s_0}; \notag
\\
\lambda(s_0)=\frac{1}{1+s_0}-\frac{p-1}{1+s_0^{p-1}}&=2-p=\lambda(0).\notag
\end{align*}

In other words,~$s_0=\rho(0)$. Thus, for any~$s\geq t \geq 0$ we have 
proved the chain of inequalities
\begin{equation*}
Q(t,s)\geq Q(0,s)\geq Q(0,\rho(0)),
\end{equation*}
which leads to
\begin{equation*}
\frac{1}{M(R)}\geq Q\big(R^{p-1},\rho(R)\big)\geq Q(0,\rho(0)) = \frac{1}{M(0)}.
\end{equation*}

Therefore,
\begin{multline*}
(c_{p,p}^{\star,\mathbb{R}})^{-1}=M(0) = \frac{1}{Q(0,s_0)}= 
\frac{(1+s_0)^p}{ (p-1)s_0^p+ps_0^{p-1}+1}
=\frac{(1+s_0)^p}{s_0+s_0^p+s_0^{p-1}+1}
=\\
\frac{(1+s_0)^{p-1}}{1+s_0^{p-1}}
=\frac{1}{(p-1)}\cdot\Big(\frac{1+s_0}{s_0}\Big)^{p-2},
\end{multline*}
which proves Theorem~\ref{Cthm} for the case of real-valued functions.

\section{Proof of Proposition~\ref{Complexc}}\label{s5}
Let~$\p \in [2,\infty)$, let~$p=\p$, $q=\frac{p}{p-1}$; the case 
$\p \in (1,2)$ is completely similar. 
\begin{defn}
Consider the function~$\BellcC\colon 
\mathbb{C}^2\times\mathbb{R}^2\mapsto \mathbb{R}$ given by the formula 
\begin{equation*}
\BellcC(z_1,z_2,x_3,x_4) = \sup\Big\{\re(\av{f\bar g}{I})\,\Big|\; 
\av{f}{I}=z_1, 
\av{g}{I}=z_2, \av{|f|^p}{I}=x_3, \av{|g|^q}{I}=x_4\Big\}.
\end{equation*}
The natural domain for~$\BellcC$ is 
\begin{equation*}
\OmegacC = \Big\{x=(z_1,z_2,x_3,x_4)\in 
\mathbb{C}^2\times\mathbb{R}^2\,\Big|\; x_3 \geq |z_1|^p, x_4 \geq |z_2|^q\Big\}.
\end{equation*}
\end{defn}

Similarly to Lemma~\ref{SimpleProperties}, the function~$\BellcC$ is 
the minimal among concave functions~$G\colon \OmegacC\to \mathbb{R}$ that 
satisfy the boundary conditions
\begin{equation*}
G(z_1,z_2,|z_1|^p,|z_2|^q) = \re(z_1\bar z_2).
\end{equation*}
We also have complex homogeneity,
\begin{equation*}
\BellcC(\zeta z_1,\zeta z_2, x_3,x_4) = \BellcC(z_1,z_2,x_3,x_4),
\quad |\zeta| = 1.
\end{equation*}
Similarly to Proposition~\ref{CFromBell},
\begin{equation*}
c_{p,r}^{\star} = \bigg(\sup_{|z_1| \in (0,1)} 
\frac{\big(\BellcC(z_1,0,1,1)\big)^r}{1 - |z_1|^{r}}\bigg)^{-1} = 
\bigg(\sup_{z_1 \in (0,1)} \frac{\big(\BellcC(z_1,0,1,1)\big)^r}{1 - z_1^{r}}
\bigg)^{-1},
\end{equation*}
the latter identity follows from the complex homogeneity above. We 
will prove the identity
\begin{equation*}
\BellcC(x) = \BellcPlus(x), \quad x\in \Omegac.
\end{equation*}
The inequality~$\BellcPlus(x) \leq \BellcC(x)$ is evident. For the 
reverse inequality, it suffices to show that the function 
\begin{equation*}
(z_1,z_2,x_3,x_4) \mapsto \BellcPlus(\re(z_1),\re(z_2),x_3,x_4)
\end{equation*}
majorizes~$\BellcC$ on~$\OmegacC$. The said function is concave, so, 
it suffices to verify the inequality on~$\FixedBoundary\OmegacC$:
\begin{equation*}
\BellcPlus(\re(z_1),\re(z_2),|z_1|^p,|z_2|^q) \geq \re(z_1\bar z_2),
\quad z_1,z_2 \in \mathbb{C}.
\end{equation*}
 
The maximal value of~$\re(z_1 \bar{z}_2)$ 
provided~$x_1=\re(z_1)$,~$x_2=\re(z_2)$,~$x_3=|z_1|^p$, 
and~$x_4=|z_2|^q$ are fixed, 
is~$x_1x_2+\sqrt{\big(x_3^{\frac{2}{p}}-x_1^2\big)
\big(x_4^{\frac{2}{q}}-x_2^2\big)}$. Thus, it suffices to prove 
\begin{equation}\label{eq2}
\BellcPlus(x_1,x_2,x_3,x_4) \geq x_1x_2+
\sqrt{\big(x_3^{\frac{2}{p}}-x_1^2\big)\big(x_4^{\frac{2}{q}}-x_2^2\big)}.
\end{equation}
We invoke Theorem~\ref{BellmanTheorem} and verify this inequality 
on each chord~$\ell_c(a_1,a_2,R)$ individually.

Pick some~$R\in[-1,1]$ and consider the chord~$\ell_c(a_1,a_2,R)$ with 
the endpoints~$a=(-1,-1,1,1)$ and~$b=(\rho(R),\phi(R),\rho(R)^p,|R|^p)$ 
(by homogeneity, we may consider such chords only). We also pick a 
point~$x=\tau a+(1-\tau)b$ on~$\ell_c(a_1,a_2,R)$, here~$\tau\in (0,1)$. 
If~$R \in [R_0,1]$, then~$\rho=R$. Therefore, by 
Theorem~\ref{BellmanTheorem},~$\BellcPlus(x)=x_3^{\frac{1}{p}}x_4^{\frac{1}{q}}$, 
and the inequality~\eqref{eq2} is simple (this is nothing more 
than~$|z_1z_2|\geq \re(z_1\bar z_2)$). So, we 
assume~$R \in [-1,R_0]$ in what follows.

By Theorem~\ref{BellmanTheorem},~$\BellcPlus(x) = 
\tau+(1-\tau)\rho(R)\phi(R)$. The coordinates of~$x$ are
\begin{align*}
x_1&=-\tau+(1-\tau)\rho,\\
x_2&=-\tau+(1-\tau)\phi,\\
x_3&=\tau+(1-\tau)\rho^p,\\
x_4&=\tau+(1-\tau)|R|^p.
\end{align*}
Thus, we may represent~\eqref{eq2} as
\begin{equation*}
\begin{split}
\tau+(1-\tau)\rho \phi \geq& 
\big((1-\tau)\rho-\tau\big)\big((1-\tau)\phi-\tau\big)+\\
&\sqrt{(\tau+(1-\tau)\rho^p)^{\frac{2}{p}}-((1-\tau)\rho-\tau)^2}
\,\,\sqrt{(\tau+(1-\tau)|R|^p)^{\frac{2}{q}}-((1-\tau)\phi-\tau)^2},
\end{split}
\end{equation*}
which might be further rewritten as
\begin{equation*}
\tau^2(1-\tau)^2(1+\rho)^2(1+\phi)^2\geq 
\bigg[(\tau+(1-\tau)\rho^p)^{\frac{2}{p}}-((1-\tau)\rho-\tau)^2\bigg]\,
\bigg[(\tau+(1-\tau)|R|^p)^{\frac{2}{q}}-((1-\tau)\phi-\tau)^2\bigg].
\end{equation*}

This is equivalent to 
\begin{equation}\label{eq3}
S_1(R,\tau) S_2(R,\tau) \leq \tau^2(1-\tau)^2, 
\end{equation}
where
\begin{equation*}
S_1(R,\tau) = \frac{(\tau+(1-\tau)\rho(R)^p)^{\frac{2}{p}}
-((1-\tau)\rho(R)-\tau)^2}{(1+\rho(R))^2},
\end{equation*}
\begin{equation*}
S_2(R,\tau) = \frac{(\tau+(1-\tau)|R|^p)^{\frac{2}{q}}
-((1-\tau)\phi(R)-\tau)^2}{(1+\phi(R))^2}.
\end{equation*}

Since~$\rho(R) \in [R_0,1]$, we have 
\begin{equation*}
S_1(\rho(R),\tau) S_2(\rho(R),\tau) \leq \tau^2(1-\tau)^2
\end{equation*}
because this inequality is equivalent to~\eqref{eq2} 
with~$\BellcPlus(x_1,x_2,x_3,x_4)$ replaced with~$x_3^{\frac1p}x_4^{\frac1q}$.
Moreover,~$S_1(R,\tau) = S_1(\rho(R),\tau)$ since~$\rho(\rho(R))=\rho(R)$. 
Therefore, it suffices to show that~$S_2(R,\tau) \leq S_2(\rho(R),\tau)$ 
when $R \in (-1,R_0]$ (note that both~$S_1$ and~$S_2$ are non-negative).

Consider the function~$S(R,\tau) =S_2(R,\tau)-S_2(\rho(R),\tau)$. 
Note that~$S_2(R,0)=S_2(\rho(R),0)=S_2(R,1)=S_2(\rho(R),1)=0$, which 
leads to~$S(R,0)=S(R,1)=0$. Let us show that the function~$S(R,\cdot)$ is 
convex on~$[0,1]$. We compute its second derivative:
\begin{equation*}
\frac{q^2}{2(2-q)} \, \frac{\partial^2}{\partial \tau^2}S(R,\tau)=
\Big(\frac{1-|R|^p}{1+\phi(R)}\Big)^2 
(\tau+(1-\tau)|R|^p)^{\frac{2}{q}-2}-
\Big(\frac{1-\rho(R)^p}{1+\rho(R)^{p-1}}\Big)^2 
(\tau+(1-\tau)\rho(R)^p)^{\frac{2}{q}-2}.
\end{equation*}
Thus, convexity of~$S(R,\cdot)$ on~$[0,1]$ is equivalent to
\begin{equation*}
\Big(\frac{1+\phi(R)}{1-|R|^p}\Big)^{p} (\tau+(1-\tau)|R|^p) 
\leq \Big(\frac{1+\rho(R)^{p-1}}{1-\rho(R)^p}\Big)^{p} (\tau+(1-\tau)\rho(R)^p).
\end{equation*}
This inequality is linear with respect to~$\tau$. We prove it at 
the endpoints~$\tau=0$ and~$\tau=1$. At these points, it turns into:
\begin{align}
\label{eq4} \Big(\frac{1+\phi(R)}{1-|R|^p}\Big)|R| &
\leq \Big(\frac{1+\rho(R)^{p-1}}{1-\rho(R)^p}\Big)\rho(R),
\\
\label{eq5} 
\frac{1+\phi(R)}{1-|R|^p} &\leq \frac{1+\rho(R)^{p-1}}{1-\rho(R)^p}.
\end{align}

First, we will show~\eqref{eq5}. Second, we will justify~$\rho(R)\geq |R|$. 
Then, inequality~\eqref{eq4} will follow from~\eqref{eq5}.

The function~$u \colon t \mapsto \frac{1+t|t|^{p-2}}{1-|t|^p}$ is 
increasing on~$(-1,1)$:
\begin{equation*}
u'(t) = |t|^{p-2}\frac{(p-1)+pt+|t|^{p}}{(1-|t|^p)^2} \geq 0.
\end{equation*}
The inequality~$u(R) \leq  u(\rho(R))$ is exactly~\eqref{eq5}. It remains to verify that~$\rho(R)\geq |R|$. 

If~$R\geq 0$, then~$R \leq R_0\leq \rho(R)$. So, we consider the case~$R \in [-1,0]$ only. 
\begin{lem}\label{Lem52}
For any~$t \geq 0$, we have~$\lambda(-t)\geq \lambda(t)$.
\end{lem}
\begin{proof}
Using the definition of~$\lambda$, we rewrite this as
\begin{equation*}
\frac{1}{1-t}-\frac{1}{1+t}=\frac{2t}{1-t^2}\geq 
\frac{2(p-1)t^{p-1}}{1-t^{2(p-1)}}=\frac{(p-1)}{1-t^{p-1}}-\frac{(p-1)}{1+t^{p-1}},
\end{equation*}
which, in its turn, is equivalent to
\begin{equation*}
1-t^{2(p-1)}-(p-1)t^{p-2}(1-t^2)\geq 0.
\end{equation*}
The left hand-side vanishes at~$t=1$ and decreases on~$t \in (0,1)$; 
here is the derivative of the left hand-side:
\begin{equation*}
-2(p-1)t^{2p-3}-(p-1)(p-2)t^{p-3}+(p-1)pt^{p-1}=-
(p-1)t^{p-3}(2t^{p}-pt^{2}+(p-2))\leq 0.
\end{equation*}
\end{proof}

Thus, we have established~$\lambda(-\rho(R))\geq \lambda(\rho(R))
=\lambda(R)$ when~$R \in (-1,0)$. By Lemma~\ref{RealR0}, the 
function~$\lambda$ is decreasing on~$[-1,R_0]$, 
consequently,~$-\rho(R) \leq R$. Therefore,~$|R|\leq \rho(R)$. The 
inequality~\eqref{eq5} is proved, and~\eqref{eq4} follows from it.

\section{The computation of~$\BelldPlus$}\label{s6}
By Corollary~\ref{GeneralConvex2}, the union of the graphs 
of~$\BelldPlus$ and~$\BelldMinus$ coincides with the boundary 
of~$\mathbb{K}$. By Corollary~\ref{DescriptionOfConvexHull}, 
the said boundary consists of the segments~$\mathfrak{L}(a_1,a_2,R)$ 
and two additional sets. Clearly, a 
point~$(x_1,\Bell_{d,\pm}(x),x_3,x_4,x_5)$,~$x \in \mathrm{int}\, \Omegad$, 
does not belong to any of these exceptional sets. Thus, each of the 
graphs of~$\BelldPlus$ and~$\BelldMinus$ on the interior of~$\Omegad$ 
consists of the segments~$\ell_d(a_1,a_2,R)$. To prove 
Theorem~\ref{BellmanTheorem2} for real-valued functions, it suffices 
to show that a segment~$\ell_d(a_1,a_2,R)$ cannot lie on the graph 
of~$\BelldPlus$ if~$a_2 < 0$: then, similarly, the 
segments~$\ell_d(a_1,a_2,r)$ with~$a_2 > 0$ do not lie on the graph 
of~$\BelldMinus$, thus, they foliate the graph of~$\BelldPlus$.

Assume the contrary: let~$a_2<0$ and let the segment~$\ell_d(a_1,a_2,R)$ 
lie on the graph of~$\BelldPlus$. Then, there exist~$t_0,t_1,t_3,t_4,t_5$ 
such that the subgraph of the affine 
function~$x\mapsto t_0 +t_1x_1+t_3x_3+t_4x_4+t_5x_5$ contains~$\mathbb{K}$ 
and its graph contains the endpoints of~$\ell_{d}(a_1,a_2,R)$. In other 
words, the inequality
\begin{equation}\label{eqn3}
\Phi_2(x_1,z):=t_0+t_1x_1+t_3|x_1|^p+t_4|z|^q+t_5x_1z-z\geq 0
\end{equation}
holds for any~$x_1, z \in \mathbb{R}$; moreover,~\eqref{eqn3} turns 
into equality at the points~$(x_1,z)=(a_1, a_2)$ 
and~$(x_1,z) = (-\rho(R)a_1, -\phi(R)a_2)$.  
The derivative of~$\Phi_2$ with respect to~$z$ vanishes at the 
points~$(x_1,z)=(a_1, a_2)$ and~$(x_1,z) = (-\rho(R)a_1, -\phi(R)a_2)$. 
Therefore, we have a system of equations
\begin{equation*}
t_5a_1+ qt_4a_2|a_2|^{q-2}-1=0,
\qquad -t_5\rho(R)a_1- qt_4a_2|a_2|^{q-2}\phi(R)|\phi(R)|^{q-2} -1=0.
\end{equation*}
Note that this system does not have solutions when~$R\in [R_0,1]$. 
Using the identity~$\phi(R)|\phi(R)|^{q-2} =R$, we solve this equations 
for~$t_4$ and~$t_5$:
\begin{equation*}
t_5 = -\frac{1+R}{(\rho(R)-R)a_1},\qquad
qt_4 = \frac{(1+\rho(R))}{(\rho(R)-R)a_2|a_2|^{q-2}}.
\end{equation*}
The function~$\rho$ is positive, moreover, for~$R \in [-1,R_0)$, 
we have~$\rho(R)>R$. Therefore, the sign of~$qt_4$ coincides with 
the sign of~$a_2$. Thus, if~$a_2<0$, then~$t_4<0$. This 
contradicts~\eqref{eqn3} for sufficiently large~$|z|$.

\section{The computation of~$d_{p,p}^{\star,\mathbb{R}}$}
\label{s7}
By Proposition~\ref{DFromBell},
\begin{equation}\label{dpp}
d_{p,p}^{\star,\mathbb{R}} = \bigg(\sup_{x_1 \in (-1,1)} 
\frac{\BelldPlus^p(x_1,1,1,0)}{1 - |x_1|^{q}}\bigg)^{-1}.
\end{equation}
Consider the segment~$\ell_d(a_1,a_2,R)$ with~$a_2>0$,~$a_1\ne 0$ 
and~$R \in [-1,R_0)$ (the point~$(x_1,1,1,0)$ lies in the interior 
of~$\Omegad$). We find the point~$x$ on this segments such that~$x_5=0$:
\begin{equation*}
(\tau+(1-\tau)\rho(R)\phi(R))a_1a_2=0 \quad \Longleftrightarrow 
\quad \tau = \frac{-\rho(R)\phi(R)}{1-\rho(R)\phi(R)}, \quad 1-\tau 
= \frac{1}{1-\rho(R)\phi(R)}.
\end{equation*}
The real~$\tau$ belongs to~$[0,1]$, which is equivalent to~$R\in [-1,0]$. 
Here are all the other coordinates of~$x$ as well as the value of 
$\BelldPlus$ at~$x$:
\begin{align}
x_1 = (\tau-(1-\tau)\rho)a_1 = -\frac{\rho(1+\phi)}{1-\rho\phi}a_1
\\
x_3 = (\tau +(1-\tau)\rho^p) |a_1|^p =\frac{\rho^p-
\rho\phi}{1-\rho\phi} |a_1|^p,
\\
x_4 = (\tau +(1-\tau)|R|^p)a_2^q = \frac{|R|^p-\rho\phi}{1-\rho\phi} a_2^q,\\
\BelldPlus(x) = (\tau -(1-\tau)\phi) a_2 = -\frac{\phi(1+\rho)}{1-\rho\phi} a_2.
\end{align}
Choosing appropriate~$a_1$ and~$a_2$ to get~$x_3 = x_4 = 1$, and 
plugging this back into~\eqref{dpp}, we see 
that~$(d_{p,p}^{\star,\mathbb{R}})^{-1}$ coincides with the maximal value 
of the function~$\tilde S$ given by the rule
\begin{equation}\label{eqn9}
\tilde S(R)=
\frac{|R|^{p-1}(1+\rho)^p(\rho^{p-1}+|R|^{p-1})^{q-1}}{(|R|+\rho)^{p-1}
\Big((1+\rho|R|^{p-1})(\rho^{p-1}+|R|^{p-1})^{q-1}-\rho(1-|R|^{p-1})^q\Big)}.
\end{equation}
\begin{lem}\label{YetAnotherMonotonicity}
The function~$\tilde{S}\colon [-1,0] \to \mathbb{R}$ attains its 
maximal value at zero.
\end{lem}
\begin{proof}
Consider the function~$S$: 
\begin{equation*}
S(u,t) = \frac{(1+ut)(u+t^{p-1}) - t(1-u)^q (u+t^{p-1})^{2-q}}{u(1+t)^p}, 
\quad (u,t) \in [0,1]^2.
\end{equation*}
We will show~$S(u,t)\geq S(0,\rho(0))$ for~$u \in [0,1]$ 
and~$t \in [\rho(0),1]$. The monotonicity wanted follows from this inequality:
\begin{equation*}
\tilde S(R)^{-1} = S(|R|^{p-1}, \rho) \cdot 
\frac{(\rho+|R|)^{p-1}}{\rho^{p-1}+|R|^{p-1}}\geq 
S(|R|^{p-1},\rho) \geq S(0,\rho(0)) = \tilde S(0)^{-1}.
\end{equation*}

Let us first show that~$\frac{\partial S}{\partial u} \geq 0$. 
We compute this derivative:
\begin{multline*}
(1+t)^p u^2 \frac{\partial S}{\partial u}= 
u^2\frac{\partial}{\partial u}\Big(ut+t^{p}+1+
\frac{t^{p-1}}{u} -t \frac{(1-u)^q}{u}(u+t^{p-1})^{2-q}\Big) = \\
t(u^2-t^{p-2})+t(1-u)^{q-1}(u+t^{p-1})^{1-q}\Big(qu(u+t^{p-1})
-(2-q)u(1-u)+(1-u)(u+t^{p-1})\Big)=\\
t(u^2-t^{p-2})+t\Big(\frac{1-u}{u+t^{p-1}}
\Big)^{q-1}(u^2+u(q-1)(1+t^{p-1})+t^{p-1}).
\end{multline*}
We make the change of variable~$v=t^{p-1}$. The positivity of the 
derivative of~$S$ with respect to~$u$ is equivalent to the positivity of
\begin{equation*}
F(u,v)=(1-u)^{q-1}(u^2+u(q-1)(1+v)+v)-(u+v)^{q-1}(v^{2-q}-u^2).
\end{equation*}

We compute~$\frac{\partial^2 F}{\partial v^2}$:
\begin{multline*}
\frac{\partial^2 F}{\partial v^2} = (2-q)(q-1)(u+v)^{q-3}(v^{2-q}-u^2)
-2(q-1)(2-q)(u+v)^{q-2}v^{1-q}+(2-q)(q-1)(u+v)^{q-1}v^{-q}=
\\
=(2-q)(q-1)(u+v)^{q-3}[v^{2-q}-u^2-2(u+v)v^{1-q}+(u+v)^2v^{-q}]
=(2-q)(q-1)(u+v)^{q-3}u^2(v^{-q}-1)\geq 0.
\end{multline*}
Therefore,~$F$ is convex with respect to~$v$. We continue the computations:
\begin{multline*}
\frac{\partial F}{\partial v}(u,1) = (1-u)^{q-1}(1+(q-1)u)
-(q-1)(1+u)^{q-2}(1-u^2)-(2-q)(1+u)^{q-1} = 
\\
u(q-1)((1+u)^{q-1}+(1-u)^{q-1})-((1+u)^{q-1}-(1-u)^{q-1}).
\end{multline*}
Further,
\begin{multline*}
\frac{1}{q-1}\frac{\partial^2 F}{\partial v \partial u}(u,1) 
= ((1+u)^{q-1}+(1-u)^{q-1})+u(q-1)((1+u)^{q-2}-(1-u)^{q-2}) 
-((1+u)^{q-2}+(1-u)^{q-2}) =
\\
=uq((1+u)^{q-2}-(1-u)^{q-2})\leq 0,
\end{multline*}
which leads to 
\begin{equation*}
\frac{\partial F}{\partial v}(u,1) 
\leq \frac{\partial F}{\partial v}(0,1) = 0.
\end{equation*}
Since~$F$ is convex with respect to~$v$, we 
have~$\frac{\partial F}{\partial v} \leq 0$. Consequently, 
\begin{multline*}
F(u,v)\geq F(u,1) = (1-u)^{q-1}(1+2(q-1)u+u^2)-(1+u)^{q-1}(1-u^2)=
\\ 
(1-u)\Big[ 2qu (1-u)^{q-2} +(1-u)^{q}-(1+u)^q\Big].
\end{multline*}
The expression in the brackets is non-negative since it is equal to 
zero at~$u=0$ and does not decrease with respect to~$u$:
\begin{multline*}
\frac{1}{q}\frac{\partial}{\partial u}\Big[ 
2qu (1-u)^{q-2} +(1-u)^{q}-(1+u)^q\Big]= 
2(1-u)^{q-2}+2u(2-q)(1-u)^{q-3}-(1-u)^{q-1}-(1+u)^{q-1} =\\
=(1+u)\Big((1-u)^{q-2} - (1+u)^{q-2}\Big) + 2u(2-q)(1-u)^{q-3}\geq 0.
\end{multline*}

Thus, we have proved~$F(u,v)\geq 0$. This leads to the 
inequality~$\frac{\partial S}{\partial u} \geq 0$. Consequently,
\begin{equation*}
S(u,t) \geq S(0,t) = (1+t)^{-p}(q-1+t^{p}+qt^{p-1}).
\end{equation*}

We compute the derivative of~$S$ with respect to~$t$:
\begin{multline*}
(1+t)^{p+1}\frac{\partial S}{\partial t}(0,t) 
= -p(q-1+t^{p}+qt^{p-1})+p(1+t)(t^{p-1}+t^{p-2}) = p(t^{p-2}+(2-q)t^{p-1}-(q-1)).
\end{multline*}
This expression increases and equals to zero at~$t=\rho(0)$ 
(see~\eqref{eqs0}). Therefore,~$S(0,t)\geq S(0,\rho(0))$ 
when~$t\geq \rho(0)$. This means we have proved that for any~$u \in [0,1]$ 
and any~$t \in [\rho(0),1]$ we have~$S(u,t) \geq S(0,\rho(0))$. 
\end{proof}

Lemma~\ref{YetAnotherMonotonicity} and the considerations before 
it lead to the proof of Theorem~\ref{Dthm} in the real-valued 
case (recall that~$s_0 = \rho(0)$):
\begin{equation*}
d_{p,p}^{\star,\mathbb{R}} = \tilde{S}(0)^{-1} = S(0,s_0) 
= \Big(\frac{s_0}{1+s_0}\Big)^{p-2}.
\end{equation*}

\section{Proof of Proposition~\ref{Complexd}}\label{s8}
We consider the case~$\p > 2$, the case~$\p < 2$ is similar. We use 
the notation~$p = \p$ and~$q = \frac{p}{p-1}$. 
Consider yet another Bellman function
\begin{equation*}
\BelldC(z_1,x_3,x_4,z_5) = 
\sup\Big\{\re(\av{g}{I})\,\Big|\;\av{f}{I}=z_1, \av{|f|^p}{I}=x_3, 
\av{|g|^q}{I}=x_4,\av{f\bar{g}}{I}=z_5\Big\}.
\end{equation*}
The natural domain of~$\BelldC$ is
\begin{equation*}
\OmegadC = \Big\{(z_1,x_3,x_4,z_5)\in 
\mathbb{C}\times\mathbb{R}^2\times\mathbb{C}\,\Big|\; 
0 \leq x_4, \, |z_1|^p\leq x_3, \, |x_5|\leq x_3^{1/p}x_4^{1/q}\Big\}.
\end{equation*}
As usual,~$\BelldC$ is minimal among concave functions 
on~$\OmegadC$ that satisfy the boundary conditions:
\begin{equation*}
\begin{aligned}
\BelldC(z_1,|z_1|^p,|z_2|^q,z_1z_2)=\re{z_2},\quad &z_1 \ne 0,\\
\BelldC(0,0,x_4,0) = x_4^{\frac1q},\quad &x_4 \geq 0.
\end{aligned}
\end{equation*}
Similar to Proposition~\ref{DFromBell},
\begin{equation*}
d_{p,r}^{\star} = \bigg(\sup_{|z_1| \in (0,1)} 
\frac{(\BelldC(z_1,1,1,0))^r}{1 - |z_1|^{\frac{qr}{p}}}\bigg)^{-1} = 
\bigg(\sup_{x_1 \in (0,1)} 
\frac{(\BelldC(x_1,1,1,0))^r}{1 - x_1^{\frac{qr}{p}}}\bigg)^{-1},
\end{equation*}
the latter identity follows from homogeneity.
Similar to Section~\ref{s5}, we will prove that
\begin{equation*}
\BelldC(x) = \BelldPlus(x), \quad x\in \Omegad.
\end{equation*}
The inequality~$\BelldC(x) \geq \BelldPlus(x)$ is evident.
Thus, it suffices to show that the function
\begin{equation*}
(z_1,x_3,x_4, z_5) \mapsto \BelldPlus(\re(z_1),x_3,x_4,\re(z_5))
\end{equation*}
majorizes~$\BelldC$ on $\OmegadC$. The said function is concave, 
which allows to verify the majorization property on the skeleton 
of~$\OmegadC$ only. This can be rewritten as
\begin{equation*}
\re(z_5/z_1)\leq \BelldPlus(\re(z_1),|z_1|^p,|z_5|^q|z_1|^{-q},\re(z_5)), 
\qquad z_1 \ne 0.
\end{equation*}
For~$\re(z_1)=x_1, \re(z_5)=x_5$,~$|z_1|$, and~$|z_5|$ fixed, the 
expression~$\re(z_5/z_1)$ attains the maximal value
\begin{equation*}
\frac{x_1x_5+\sqrt{(|z_1|^2-x_1^2)(|z_5|^2-x_5^2)}}{|z_1|^2}.
\end{equation*}
Thus, it remains to prove
\begin{equation}\label{eqn7}
\BelldPlus(x_1,x_3,x_4,x_5) \geq 
\frac{x_1x_5+\sqrt{(x_3^{2/p}-x_1^2)(x_4^{2/q}x_3^{2/p}-x_5^2)}}{x_3^{2/p}}, 
\quad x_3 \ne 0.
\end{equation}

\begin{lem}\label{lem67}
For any $x \in \Omega_d$ 
\begin{equation}\label{eqn8}
(\Bell_{d,\pm}(x) x_3^{2/p}-x_1x_5)^2 \geq (x_3^{2/p}-x_1^2)(x_4^{2/q}x_3^{2/p}-x_5^2).
\end{equation}
\end{lem}
\begin{proof}
Let us temporarily use the notation~$b = \Bell_{d,\pm}$. With this 
notation,~\eqref{eqn8} turns into
\begin{equation*}
b^2 x_3^{2/p}- 2x_1x_5b \geq x_4^{2/q}x_3^{2/p} - x_4^{2/q}x_1^2 -x_5^2,
\end{equation*}
which is equivalent to
\begin{equation}\label{Turn}
(x_5-x_1b)^2\geq (x_3^{2/p}-x_1^2)(x_4^{2/q}-b^2).
\end{equation}
The inequality~\eqref{eq2}, together with~\eqref{SymmetryFormula}, leads to
\begin{equation*}
\BellcMinus(x_1,x_2,x_3,x_4) \leq x_1x_2 - 
\sqrt{(x_3^{2/p}-x_1^2)(x_4^{2/q}-x_2^2)}.
\end{equation*}
Thus, the inequality 
\begin{equation*}
(x_5-x_1x_2)^2 \geq (x_3^{2/p}-x_1^2)(x_4^{2/q}-x_2^2)
\end{equation*}
holds for any point~$x= (x_1,x_2,x_3,x_4,x_5)$ that lies on the union 
of the graphs of $\BellcPlus$ and $\BellcMinus$, which is the 
same as~\eqref{Turn}.
\end{proof}

Let us consider two functions 
\begin{equation*}
G^-(x) = \frac{x_1x_5-
\sqrt{(x_3^{2/p}-x_1^2)(x_4^{2/q}x_3^{2/p}-x_5^2)}}{x_3^{2/p}}, \qquad
G^+(x) = \frac{x_1x_5+\sqrt{(x_3^{2/p}-x_1^2)(x_4^{2/q}x_3^{2/p}-x_5^2)}}{x_3^{2/p}}
\end{equation*}
defined on the interior of~$\Omegad$. The inequality~\eqref{eqn7} 
follows from~$\BelldMinus \leq G^-< G^+\leq \BelldPlus$.
Lemma~\ref{lem67} says that $\Bell_{d,\pm}(x) \notin (G^-(x),G^+(x))$ 
for any~$x \in \Omegad$ such that~$x_3 \ne 0$. The interior of~$\Omegad$ 
is connected, the functions~$\Bell_{d,\pm}$ and $G^\pm$ are continuous 
on it, and moreover,~$G^+>G^-$. Therefore, either~$\BelldPlus\geq G^+$, 
or~$\BelldPlus\leq G^-$ everywhere. It remains to notice
\begin{equation*}
\BelldPlus(0,0,1,1)\geq \BelldMinus(0,0,1,1)=-\BelldPlus(0,0,1,1),
\end{equation*}
which shows,~$\BelldPlus(0,0,1,1)\geq 0> G^-(0,0,1,1)$. 
Thus,~$\BelldPlus^+\geq G^+$ and~\eqref{eqn7} is proved.

\section{The computation of~$c_{q,r}^{\star}$ 
and~$d_{q,r}^{\star}$ for~$q < 2$}
\label{s9}

\subsection{The computation of $c_{q,r}^{\star}$}

Similar to Proposition~\ref{CFromBell},
\begin{equation}\label{Supremum}
c_{q,r}^{\star,\mathbb{R}} = 
\bigg(\sup_{x_2 \in (-1,1)} 
\frac{\BellcPlus^r(0,x_2,1,1)}{1 - |x_2|^{r}}\bigg)^{-1}=
\bigg(\sup_{|x_2|^q < x_4 } 
\frac{\BellcPlus^r(0,x_2,x_3,x_4)}{(x_4^{r/q} - |x_2|^{r})x_3^{r/p}}\bigg)^{-1}.
\end{equation}

On each segment~$\ell_c(a_1,a_2,R)$, we find a 
point~$x=\tau a + (1-\tau)b$ with~$x_1 = 0$. In other 
words,~$(\tau - (1-\tau)\rho)=0$, i.e.~$\tau = \frac{\rho}{1+\rho}$. 
Here are the other coordinates of~$x$ and the value of~$\BellcPlus$ there: 
\begin{align*}
x_2 &= a_2 \frac{\rho-\phi}{1+\rho},& x_3 &
= |a_1|^p \frac{\rho+\rho^p}{1+\rho}\\
x_4 &= |a_2|^q \frac{\rho+|R|^p}{1+\rho},& \BellcPlus(x) &= a_1a_2\frac{\rho+\rho\phi}{1+\rho}. 
\end{align*}
We plug these values back into~\eqref{Supremum}:
\begin{equation*}
\frac{1}{c_{q,r}^{\star,\mathbb{R}}} = 
\sup_{R \in (-1,1)} 
\frac{\rho^r (1+\phi)^r}{(1+\rho)^{r/q}(\rho+\rho^p)^{r/p}
\Big[\big(\frac{\rho+|R|^p}{1+\rho}\big)^{r/q}- 
\big|\frac{\rho-\phi}{1+\rho}\big|^r\Big]}.
\end{equation*}

By Lemma~\ref{rlessthan2}, the latter supremum equals~$+\infty$ provided~$r<2$. 

We claim without proof that the point~$R=-1$ is the global maximum 
provided~$r=2$. The limiting value at this point is~$(p-1)$. 
Thus,~$c_{q,2}^{\star,\mathbb{R}}=(q-1)$. 

We also claim that~$0$ is the absolute maximum when~$r=p$. In this case,
\begin{equation*}
c_{q,p}^{\star,\mathbb{R}} = \frac{1+\rho(0)^{p-1}}{1+\rho(0)}.
\end{equation*}

\newcommand{\tbdp}{\tilde\Bell_{d,+}}
\newcommand{\tbdm}{\tilde\Bell_{d,-}}
\newcommand{\tod}{\tilde\Omega_d}

\subsection{Proof of Theorem~\ref{Elementary}}
\begin{prop}
For any~$p > 2$ we have~$d_{p,2(p-1)}^{\star} = 1$.
\end{prop}
\begin{proof}
Let us first prove that~$d_{p,2(p-1)}^{\star,\mathbb{R}} = 1$. In other words, 
we are going to prove the inequality
\begin{equation*}
\big|\langle \Nop_p(f),e\rangle\big|^2+ 
\inf_{\alpha}\|f+\alpha e\|_{L^p}^{2p-2} \leq \|f\|_{L^p}^{2p-2}
\end{equation*}
for real-valued functions~$f$ and~$e$, with the 
assumption~$\|e\|_{L^p}=1$. Let the infimum be attained 
at~$\alpha=\alpha^{\star}$.
 
We consider the function~$F\colon\mathbb{R}\to\mathbb{R}$ given by the rule
\begin{equation}\label{Real-Valued}
F(t) = \|f+te\|_{L^p}^{2p-2}-\big|\langle \Nop_p(f+te),e\rangle\big|^2
\end{equation}
and compute its derivative:
\begin{equation*}
F'(t) = 2(p-1) \langle \Nop_p(f+te),e\rangle \Big(\|f+te\|_{L^p}^{p-2} 
- \int |f+te|^{p-2}e^2 \Big).
\end{equation*}
The expression in the parentheses is non-negative by H\"older's inequality. 
Moreover, the function 
\begin{equation*}
t\mapsto \langle \Nop_p(f+te),e\rangle
\end{equation*}
has positive derivative~$(p-1)\int |f+te|^{p-2}e^2$, in particular, 
this function has no more than one root. Thus, the value~$t=\alpha^{\star}$ 
is its unique root. The function~$F$ attains its minimal value 
at~$t=\alpha^{\star}$. What is more,
\begin{equation*}
F(\alpha^{\star}) = \|f+\alpha^{\star} e\|_{L^p}^{2p-2},
\end{equation*}
which makes the inequality~\eqref{Real-Valued} equivalent 
to~$F(0)\geq F(\alpha^{\star})$.

We return to the complex-valued case. It suffices to prove a 
slightly weaker inequality
\begin{equation*}
\big|\re\langle \Nop_p(f),e\rangle\big|^2+ 
\inf_{\alpha}\|f+\alpha e\|_{L^p}^{2p-2} \leq \|f\|_{L^p}^{2p-2}, 
\quad \|e\|_{L^p} = 1.
\end{equation*}
The inequality~\eqref{eq-Hold4} (with~$r=2(p-1)$ and~$d_{p,r}=1$) 
follows if one multiplies~$e$ by suitable scalar to make the scalar 
product~$\langle \Nop_p(f),e\rangle$ real. We slightly strengthen our 
inequality:
\begin{equation*}
\big|\re\langle \Nop_p(f),e\rangle\big|^2+ 
\inf_{t\in\mathbb{R}}\|f+t e\|_{L^p}^{2p-2} \leq 
\|f\|_{L^p}^{2p-2}, \quad \|e\|_{L^p} = 1.
\end{equation*}
This inequality can be proved with the help of a modified functions~$F$,
\begin{equation*}
F(t) = \|f+te\|_{L^p}^{2p-2}-\big|\re\langle \Nop_p(f+te),e\rangle\big|^2,
\end{equation*}
exactly the same way as~\eqref{Real-Valued}.
\end{proof}

\begin{lem}\label{Duality}
For dual exponents $p$ and $q$, the identities~$d_{p,r(p-1)}^{\star} = 1$ 
and~$d_{q,r}^{\star} = 1$ are equivalent.
\end{lem}
\begin{proof}
By the Hahn-Banach theorem, we have $d_{p,r(p-1)}^{\star} = 1$ if and only if the 
inequality 
\begin{equation*}
\big|\langle \Nop_p(f),e
\rangle\big|^{r}+
\big|\langle f,g\rangle\big|^{r(p-1)}
\le\|f\|^{r(p-1)}_{L^p},\qquad 
\|e\|_{L^p}=\|g\|_{L^q}=1, \quad \langle e,g\rangle = 0
\end{equation*}
holds true. We introduce a new function~$F = \Nop_p(f)$ and restate this as
\begin{equation*}
\big|\langle F,e\rangle\big|^{r}+
\big|\langle \Nop_q(F),g\rangle\big|^{r(p-1)}
\le\|F\|^r_{L^q},\qquad 
\|e\|_{L^p}=\|g\|_{L^q}=1, \quad \langle e,g\rangle = 0,
\end{equation*}
which is equivalent to~$d_{q,r}^{\star}=1$.
\end{proof}
\begin{cor}
For any~$q < 2$, we have~$d_{q,2}^{\star} = 1$.
\end{cor}
By Lemma~\ref{rlessthan2},~$d^\star_{q,r} = 0$ holds when~$r < 2$. 
Hence, by Lemma~\ref{Duality}, we have $d^\star_{p,r} < 1$ if~$r < 2(p-1)$ 
and~$p >2$. This proves Theorem~\ref{Elementary}.

\noindent Haakan Hedenmalm

\noindent Department of Mathematics,
KTH Royal Institute of Technology,
Sweden.

\noindent Department of Mathematics and Mechanics,
St-Petersburg State University, 28 Universitetski pr., St-Petersburg 198504,
Russia.

\noindent haakanh@math.kth.se

\medskip

\noindent Dmitriy M. Stolyarov

\noindent Department of Mathematics, Michigan State University, USA.

\noindent P. L. Chebyshev Research Laboratory, 
St. Petersburg State University, Russia.

\noindent St. Petersburg Department of Steklov Mathematical Institute, 
Russian Academy of Sciences (PDMI RAS), Russia.

\noindent dms@pdmi.ras.ru

\medskip

\noindent Vasily I. Vasyunin

\noindent St. Petersburg Department of Steklov Mathematical Institute, 
Russian Academy of Sciences (PDMI RAS), Russia.

\noindent Department of Mathematics and Mechanics,
St-Petersburg State University, 28 Universitetski pr., St-Petersburg 198504,
Russia.

\noindent vasyunin@pdmi.ras.ru

\medskip

\noindent Pavel B. Zatitskiy

\noindent D\'epartement de math\'ematiques et applications, \'Ecole normale 
sup\'erieure, CNRS, PSL Research University,
France.

\noindent P. L. Chebyshev Research Laboratory, St. Petersburg State 
University, Russia.

\noindent St. Petersburg Department of Steklov Mathematical Institute, 
Russian Academy of Sciences (PDMI RAS), Russia.

\noindent pavelz@pdmi.ras.ru


\begin{thebibliography}{9}
\bibitem{BCL} K. Ball, E. Carlen, E. Lieb, \emph{Sharp uniform convexity 
and smoothness inequalities for trace norms}. Invent. Math. 
{\bf 115} (1994), 463-482.

\bibitem{BO} R. Ba\~nuelos, A. Os\c{e}kowski, \emph{Stability in 
Burkholder's differentially subordinate martingales inequalities and 
applications to Fourier multipliers}. 
https://arxiv.org/abs/1609.08672.

\bibitem{BarHed} A. Baranov, H. Hedenmalm, \emph{Boundary properties 
of Green functions in the plane}. Duke Math. J. {\bf145} (1) (2008), 1-24. 

\bibitem{Carlen} E. Carlen, \emph{Duality and stability for functional 
inequalities}.  Ann. Fac. Sci. Toulouse Math. (6) {\bf 26} (2017), no. 2,
319-350.

\bibitem{Christ1} M. Christ, \emph{A sharpened Hausdorff-Young inequality}. 
https://arxiv.org/abs/1406.1210.

\bibitem{Christ2} M. Christ, \emph{A sharpened Riesz-Sobolev inequality}.
https://arxiv.org/abs/1706.02007.

\bibitem{Clarkson}	
 J.~A.~Clarkson,
\emph{Uniformly convex spaces}. Trans. Amer. Math. Soc. {\bf 40} (1936), 
396-414.

\bibitem{Hanner}
 O.~Hanner,
\emph{On the uniform convexity of $L^p$ and $l^p$}. Ark. Mat. {\bf 3} (1956), 
239-244.

\bibitem{ShortReport} P. Ivanishvili, N. N. Osipov, 
D. M. Stolyarov, V. I. Vasyunin, P. B. Zatitskiy, 
\emph{On Bellman function for extremal problems in~$\BMO$}. 
C. R. Math. Acad. Sci. Paris {\bf 350} (2012), no. 11-12, 561-564.

\bibitem{TAMS} P. Ivanishvili, N. N. Osipov, D. M. Stolyarov, 
V. I. Vasyunin, Pavel B. Zatitskiy, \emph{Bellman function for 
extremal problems in~$\BMO$}.  Trans Amer. Math. Soc. {\bf 368} (2016), no. 5,
3415-3468.

\bibitem{ShortReport2} P. Ivanisvili, N. N. Osipov, 
D. M. Stolyarov, V. I. Vasyunin, P. B. Zatitskiy, 
\emph{Sharp estimates of integral functionals on classes of functions 
with small mean oscillation}. C. R. Math. Acad. Sci. Paris 
{\bf 353} (2015), no. 12, 1081-1085.

\bibitem{Mem} P. Ivanisvili, D. M. Stolyarov, V. I. Vasyunin, 
P. B. Zatitskiy, \emph{Bellman function for extremal problems in 
BMO \textup{II:} evolution}. Mem. Amer. Math. Soc., to appear.

\bibitem{ISZ} P. Ivanisvili, D. M. Stolyarov, P. B. Zatitskiy, 
\emph{Bellman VS Beurling\textup: sharp estimates of uniform 
convexity for $L^p$ spaces}.  
St. Petersburg Math. J. {\bf 27} (2016), no. 2, 333-343.

\bibitem{Holder} O. H\"older, \emph{\"Uber einen Mittelwertsatz}, 
Nachrichten von der K\"onigl. Gesellschaft der Wissenschaften und der 
Georg-Augusts-Universit\"at zu G\"ottingen, Band {\bf 2} (1889), 
38-47 (in German).

\bibitem{NTV} F.~Nazarov, S.~Treil, A.~Volberg, \emph{
Bellman function in stochastic optimal control and harmonic analysis}.
Systems, approximation, singular integral operators, and related topics 
(Bordeaux, 2000), 393-423. Oper. Theory Adv. Appl.~{\bf 129}, Birkh\"auser 
Basel, 2001.

\bibitem{Osekowski} A.~Os\c{e}kowski,
\emph{Sharp Martingale and Semimartingale Inequalities}.
Monografie Matematyczne IMPAN~{\bf 72}, Springer-Verlag, Basel, 2012.
	
\bibitem{Rogers} L. J. Rogers, \emph{An extension of a
 certain theorem in inequalities}. Messenger of Mathematics, 
New Series {\bf 17} (1888), no. 10, 145-150.

\bibitem{Shapbook} H. S. Shapiro, \emph{Topics in approximation theory}.
With appendices by Jan Boman and Torbj\"orn Hedberg. 
Lecture Notes in Math., Vol. \textbf{187}. Springer-Verlag, Berlin-New York, 
1971.

\bibitem{SV} L. Slavin, V. Vasyunin, \emph{Cincinnati lectures on 
Bellman functions}. https://arxiv.org/abs/1508.07668.

\bibitem{Volberg} A.~Volberg, \emph{Bellman function technique in 
Harmonic Analysis}. Lectures of INRIA Summer School in Antibes, 
June 2011, http://arxiv.org/abs/1106.3899.

\end{thebibliography}
\end{document}